\documentclass{article}

\usepackage{latexsym,amssymb,amsfonts,amsmath,amsthm,ifthen,mathrsfs}

\theoremstyle{plain}
\newtheorem{thm}{Theorem}[section]
\newtheorem{cor}[thm]{Corollary}
\newtheorem{lem}[thm]{Lemma}
\newtheorem{prop}[thm]{Proposition}

\theoremstyle{remark}
\newtheorem{rem}{Remark}[section]

\DeclareMathOperator{\Dom}{Dom}
\DeclareMathOperator{\im}{Im}
\DeclareMathOperator{\Span}{Sp}
\DeclareMathOperator{\Curl}{curl}
\DeclareMathOperator{\Div}{div}
\DeclareMathOperator{\interior}{int}
\DeclareMathOperator{\supp}{supp}

\newcommand{\loc}{\mathrm{loc}}

\newcommand{\R}{\mathbb{R}}
\newcommand{\N}{\mathbb{N}}
\newcommand{\NZ}{\mathbb{N}_0}
\newcommand{\C}{\mathbb{C}}
\newcommand{\Z}{\mathbb{Z}}
\renewcommand{\epsilon}{\varepsilon}

\newcommand{\ipd}[2]{\langle{#1},{#2}\rangle}
\newcommand{\abs}[1]{\lvert{#1}\rvert}
\newcommand{\lrabs}[1]{\left\lvert{#1}\right\rvert}
\newcommand{\bigabs}[1]{\bigl\lvert{#1}\bigr\rvert}
\newcommand{\norm}[1]{\lVert{#1}\rVert}
\newcommand{\bignorm}[1]{\bigl\lVert{#1}\bigr\rVert}

\renewcommand{\d}{\,d}

\newcommand{\ov}{\overline}
\newcommand{\wt}{\widetilde}

\newcommand{\reg}{\Omega}  
\newcommand{\rega}[1][0]{\reg_{#1}}
\newcommand{\cir}{{S^1}}  

\newcommand{\disc}[1][]{\ifthenelse{\equal{#1}{}}{\mathbb{D}}{\mathbb{D}_{#1}}} 
\newcommand{\cdisc}[2]{\ov{\mathbb{D}}_{#1}(#2)}

\newcommand{\flux}[2][\reg]{\Phi_{#1}(#2)}

\newcommand{\magp}{A}
\newcommand{\magf}{B}
\newcommand{\bndB}{\beta}  
\newcommand{\smp}{\phi}  

\newcommand{\LA}{\mathscr{A}}

\newcommand{\mm}[2][]{\ifthenelse{\equal{#1}{}}{P_{#2}}{P^{#1}_{#2}}}
\newcommand{\Pauli}[2][]{\ifthenelse{\equal{#1}{}}{\mathcal P_{#2}}{\mathcal P^{#1}_{#2}}}
\newcommand{\Dirac}[1]{\mathcal D_{#1}}
\newcommand{\Sch}[1]{H_{#1}}

\newcommand{\sob}[2][0]{W^{#2}_{#1}}
\newcommand{\msob}[1][\magp]{\mathcal H_{#1}} 
\newcommand{\cont}[1][]{C^{#1}}  
\newcommand{\hardy}[2][]{H^2_{#1}(#2)}
\newcommand{\hol}[1][\disc]{\mathcal{O}(#1)}  
\newcommand{\ahol}{\ov{\mathcal{O}}(\disc)}  
\newcommand{\ext}{E}  

\newcommand{\gcf}[2]{\ifthenelse{\equal{#2}{}}{\mathsf{N}_{#1}}{\mathsf{N}_{#1}(#2)}}
\newcommand{\biggcf}[2]{\mathsf{N}_{#1}\bigl(#2\bigr)}
\newcommand{\ncf}[2]{\mathsf{n}_{#1}(#2)}

\newcommand{\sq}[2][k]{S^{#1,#2}}  
\newcommand{\sqI}[1][k]{J_{#1}}  
\newcommand{\msq}[2][k]{\widehat{S}^{#1,#2}}  
\newcommand{\csq}[2][k]{\ov{S}^{#1,#2}}  
\newcommand{\Usq}[1][k]{\reg^{#1}} 
\newcommand{\cmagf}[1][k,j]{b^{#1}}
\newcommand{\amagf}[1][k]{\magf^{#1}}
\newcommand{\amagp}[1][k,j]{\magp^{#1}}
\newcommand{\wtamagp}[1][k,j]{\wt{\magp}^{#1}}
\newcommand{\sco}[1]{\ifthenelse{\equal{#1}{0}}{\phi}{\psi_{#1}}}  
\newcommand{\dbsq}[1][\delta]{\wt{R}^{k}_{#1}}  
\newcommand{\drsq}[1][k]{R^{#1}_{\delta}}  
\newcommand{\ubeval}[1][k]{\Gamma^+_{#1}}
\newcommand{\lbeval}[1][k]{\Gamma^-_{#1}}

\newcommand{\sep}{\sigma} 
\newcommand{\af}[1][f]{w_{#1}}
\newcommand{\bvaf}[1][f]{\omega_{#1}}
\newcommand{\lab}[1][b]{\ell_{#1}}
\newcommand{\Diwef}[1]{\xi_{#1}}  
\newcommand{\sfiw}[1]{\alpha_{#1}}
\newcommand{\sfiv}[1]{\beta_{#1}}
\newcommand{\stf}[1][\delta]{\ifthenelse{\equal{#1}{}}{\rho}{\rho_{#1}}}  
\newcommand{\cof}[1][]{\chi_{#1}}  
\newcommand{\dnb}{h}  
\newcommand{\bndh}{\kappa}  
\newcommand{\idnb}[1][]{\eta_{#1}}  
\newcommand{\iidnb}[1][]{\zeta_{#1}}  
\newcommand{\ciw}[1][\tau]{y_{#1}}
\newcommand{\civ}[1][\tau]{z_{#1}}
\newcommand{\oproj}[1]{Q_{#1}}
\newcommand{\indX}[1][]{\ifthenelse{\equal{#1}{}}{\mathcal{K}}{\mathcal{K}_{#1}}}

\newcommand{\fc}[2]{\widehat{#1}(#2)}  
\newcommand{\hproj}[1]{\Pi^{#1}}  

\newcommand{\Co}[1]{C_{#1}}  

\title{Approximate Zero Modes for the Pauli Operator on a Region}
\author{Daniel M.~Elton\\[6pt] Department of Mathematics and Statistics\\ Fylde College\\ Lancaster University\\ Lancaster LA1 4YF, United Kingdom\\[4pt] e-mail: d.m.elton@lancaster.ac.uk}

\begin{document}

\maketitle

\begin{abstract}
Let $\Pauli{\reg,t\magp}$ denoted the Pauli operator on a bounded open region $\reg\subset\R^2$ with Dirichlet boundary conditions 
and magnetic potential $\magp$ scaled by some $t>0$.
Assume that the corresponding magnetic field $\magf=\Curl\magp$ satisfies $\magf\in L\log L(\reg)\cap\cont[\alpha](\rega)$ 
where $\alpha>0$ and $\rega$ is an open subset of $\reg$ of full measure (note that, the Orlicz space $L\log L(\reg)$ contains $L^p(\reg)$ for any $p>1$). 
Let $\gcf{\reg,t\magp}{\lambda}$ denote the corresponding eigenvalue counting function. 
We establish the strong field asymptotic formula
\[
\gcf{\reg,t\magp}{\lambda(t)}\,=\,\frac{t}{2\pi}\int_{\reg}\abs{\magf(x)}\d x\;+o(t)
\]
as $t\to+\infty$, whenever $\lambda(t)=Ce^{-ct^\sep}$ for some $\sep\in(0,1)$ and $c,C>0$.
The corresponding eigenfunctions can be viewed as a localised version of the Aharonov-Casher zero modes for the Pauli operator on $\R^2$.

\bigskip

\end{abstract}

\section{Introduction}

Let $\reg\subset\R^2$ be a bounded open region and $\magp=(\magp_1,\magp_2)\in L^2_\loc(\reg,\R^2)$ a magnetic potential.
The corresponding magnetic momentum operator is then $\mm{\magp}=-i\nabla-\magp$, where $\nabla=(\nabla_1,\nabla_2)$ denotes the gradient operator on $\R^2$. 
We wish to consider the Pauli operator $\Pauli{\reg,\magp}$ on $\reg$ with magnetic potential $A$. For Dirichlet boundary conditions this can 
be defined as the non-negative operator $\Pauli{\reg,\magp}$ associated to the closure of the form
\begin{equation}
\label{Pauliformdef1:eq}
\mathbf{p}_{\reg,\magp}(u)=\norm{\mm{\magp,+}u_+}^2+\norm{\mm{\magp,-}u_-}^2,
\quad u=\begin{pmatrix}u_+\\u_-\end{pmatrix}\in\cont[\infty]_0(\reg,\C^2),
\end{equation}
where $\mm{\magp,\pm}=\mm{\magp,1}\pm i\mm{\magp,2}$. 
Set $\magf=\Curl\magp=\nabla_1\magp_2-\nabla_2\magp_1$, the magnetic field associated with the potential $\magp$ (initially defined as a distribution).
A straightforward formal calculation leads to the Lichnerowicz formula
\begin{equation}
\label{PauliMagSch:eq}
\Pauli{\reg,\magp}=\begin{pmatrix}\Sch{\reg,\magp}&0\\0&\Sch{\reg,\magp}\end{pmatrix}-\begin{pmatrix}B&0\\0&-B\end{pmatrix},
\end{equation}
where $\Sch{\reg,\magp}=\mm[2]{\magp,1}+\mm[2]{\magp,2}$ is the magnetic Schr\"odinger operator. 
 If we assume that $\magf$ belongs to the Orlicz space $L\log L(\reg)$ then \eqref{PauliMagSch:eq} can be rigorously justified and used to help show that $\Pauli{\reg,\magp}$ has a compact resolvent and hence discrete spectrum (see Proposition \ref{PauliMagSch:prop}).
Enumerate the eigenvalues of $\Pauli{\reg,\magp}$ (including multiplicities) as $0\le\lambda_1(\Pauli{\reg,\magp})\le\lambda_2(\Pauli{\reg,\magp})\le\dots$, and introduce the corresponding counting function
\[
\gcf{\reg,\magp}{\lambda}=\#\bigl\{n\in\N:\lambda_n(\Pauli{\reg,\magp})\le\lambda\bigr\},
\quad\lambda\in\R.
\]
We are interested in the behaviour of $\gcf{\reg,\magp}{}$ in the strong field regime. Fixing $\magp$ we consider $\gcf{\reg,t\magp}{\lambda(t)}$ for the scaled potential $t\magp$ and $\lambda(t)\le O(t)$ in the limit $t\to+\infty$. A simple rescaling shows that this is equivalent to the semi-classical regime.

When $\lambda(t)=O(t)$ the quantity $\gcf{\reg,t\magp}{\lambda(t)}$ obeys a natural Weyl type asymptotics. 
To state this precisely introduce auxiliary functions $\nu^-$ and $\nu^+$ which are, respectively, the maximal lower and minimal upper semi-continuous extensions of
\begin{equation}
\label{defnu:eq}
\nu(b,\lambda)=\frac{\abs{b}}{2\pi}\,\#\bigl\{m\in\Z:2\abs{mb}\le\lambda\bigr\},
\quad\lambda,b\in\R,\,b\neq0,\,\lambda\notin2\abs{b}\NZ.
\end{equation}

\begin{thm}
\label{mainres1:thm}
Suppose $\magf\in L\log L(\reg)\cap\cont(\rega)$ where $\rega\subseteq\reg$ is open and $\reg\setminus\rega$ has zero (Lebesgue) measure. 
If $\lambda(t)=\Lambda t+o(t)$ for some $\Lambda\in\R$ then
\[
\liminf_{t\to\infty}\frac1t\,\gcf{\reg,t\magp}{\lambda(t)}
\,\ge\,\int_{\reg}\nu^-(\magf(x),\Lambda)\d x
\]
and
\[
\limsup_{t\to\infty}\frac1t\,\gcf{\reg,t\magp}{\lambda(t)}
\,\le\,\int_{\reg}\nu^+(\magf(x),\Lambda)\d x.
\]
\end{thm}

Numerous results similar or related to Theorem \ref{mainres1:thm} have been obtained. 
Some of the earliest work (\cite{CdV}, \cite{Ta}) looked at spectral asymptotics for magnetic (Schr\"{o}\-din\-ger) bottles. 
While these works focused on a different class of operators the ideas of \cite{CdV} in particular form the basis of our approach to Theorem \ref{mainres1:thm} (see also \cite{Tr}). For magnetic Schr\"{o}dinger operators on a region various two term spectral asymptotic questions 
have been considered in both the Dirichlet and Neumann cases (see \cite{HM}, \cite{FH}, \cite{CFFH}, \cite{F} and references therein); whilst giving 
more precise details, these results also require greater regularity (and other conditions) on $\reg$ and $\magf$.

In another direction, various authors have considered bound states of the Pauli operator with an additional electric potential. The presence of the latter distinguishes the strong field and semi-classical regimes, leading to multi-parameter problems. The semi-classical behaviour of sums of negative eigenvalues of the form $\sum_{n}\abs{\lambda_n}^\gamma$, $\gamma>0$ was considered in \cite{LSY2}, \cite{ES2}, \cite{Sob2} and \cite{ES4} for example. These works all rely on a priori bounds on the eigenvalue sums which have the correct order in the parameters; typically Lieb-Thirring type inequalities have been developed for this purpose. However eigenvalue counting corresponds to the case $\gamma=0$ (also known as the CLR inequality) and is always excluded in dimension $2$.

The asymptotic bounds in Theorem \ref{mainres1:thm} remain finite provided $\magf\in L^1(\reg)$.
We use the slightly stronger condition $\magf\in L\log L(\reg)$ to obtain a priori bounds on $\gcf{\reg,t\magp}{\lambda(t)}$ (covering the lack of a suitable Lieb-Thirring/CLR inequality); in turn these bounds are derived from estimates in \cite{Sol} which don't extend to cover the $L^1$ case. The continuity condition $\magf\in\cont(\reg_0)$ relates to our method for approximating $\magf$ locally by a constant field. While it is likely that at least the latter condition can be relaxed the optimal regularity condition for $\magf$ remains unclear.

\medskip

The asymptotic lower and upper bounds given by Theorem \ref{mainres1:thm} differ if the set
\begin{equation}
\label{contlandaulevel:eq}
\bigl\{x\in\reg:\text{$2m\abs{\magf(x)}=\Lambda$ for some $m\in\NZ$}\bigr\}
\end{equation}
has non-zero measure. When $\Lambda\neq0$ this is a non-generic situation for variable fields. 
On the other hand, when $\lambda(t)=o(t)$ \eqref{contlandaulevel:eq} is the whole of $\reg$ for any $\magf$;
the lower bound in Theorem \ref{mainres1:thm} then reduces to $0$
while the upper bound becomes
\begin{equation}
\label{mainres1lambda0:eq}
\limsup_{t\to\infty}\frac1t\,\gcf{\reg,t\magp}{\lambda(t)}
\,\le\,\flux{\abs{\magf}},
\end{equation}
where
\[
\flux{b}=\frac1{2\pi}\int_{\reg}b(x)\,\d x
\]
is the \emph{flux} of a magnetic field $b$ on $\reg$ (see the end of Section \ref{OpBackGnd:sec} for some further details). It transpires that the upper bound gives the correct asymptotics for even sub-exponentially decaying $\lambda(t)$. Our main result is the following (in which $\cont[\alpha]$ is used to denote the space of H\"{o}lder continuous functions). 

\begin{thm}
\label{mainres2:thm}
Suppose $\magf\in L\log L(\reg)\cap\cont[\alpha](\rega)$ where $\alpha>0$, $\rega\subseteq\reg$ is open and $\reg\setminus\rega$ has zero (Lebesgue) measure. 
If $\lambda(t)\ge Ce^{-ct^\sep}$ for some constants $\sep\in(0,1)$ and $c,C>0$ then
\[
\liminf_{t\to\infty}\frac1t\,\gcf{\reg,t\magp}{\lambda(t)}
\,\ge\,\flux{\abs{\magf}}.
\]
\end{thm}

For strong fields this result guarantees the existence of approximately $\flux{\abs{t\magf}}$ sub-exponentially small eigenvalues of the Pauli operator. The corresponding eigenfunctions, which we informally term \emph{approximate zero modes}, can be viewed as a local version of the Aharonov-Casher zero modes. The latter are a dimension $\lfloor\abs{\flux[\R^2]{t\magf}}\rfloor$ set of spin-definite zero energy bound states of the Pauli operator on $\R^2$; the spin is aligned with the dominate sign of $\magf$ (see \cite{AC} and \cite{EV}). In the strong field limit strong localisation should confine such states to regions where $\magf$ has its dominate sign; indeed such localisation of the Aharonov-Casher construction lies at the heart of our argument (see below and Section \ref{AZM:sec} for further details). A different manifestation of this localisation, relating to the ground state density of the Pauli operator on $\R^2$, was obtained in \cite{E}.

\begin{rem}
The magnetic potential $\magp$ (and hence the Pauli operator $\Pauli{\reg,\magp}$) is not uniquely defined by $\magf$. If $\reg$ is simply connected  different choices of $\magp$ lead to unitarily equivalent Pauli operators (see Proposition \ref{PauliUe:prop}) so the counting function $\gcf{\reg,\magp}{}$ will not depend on the particular choice of $\magp$. For more general regions this is no longer true; in this case our results hold independently of the choice of $\magp$. 
\end{rem}

\begin{rem}
It is possible to consider operators corresponding to non-Dirichlet boundary conditions.
While there is a natural choice for a Neumann version of the magnetic Schr\"{o}dinger operator 
(which has received particular attention in connection with the Ginzburg-Landau theory of superconductivity),
it is less clear how one should define a Neumann version of the Pauli operator. One possibility would be to use the maximal closed extensions of $\mm{\magp,\pm}$ in \eqref{Pauliformdef1:eq}; however such an operator does not have a compact resolvent (even when $\magp\equiv0$), leading to a very different class of spectral problems. Alternatively one could use \eqref{PauliMagSch:eq} to define a ``Neumann'' Pauli operator in terms of the Neumann magnetic Schr\"{o}dinger operator. 
With some additional restrictions on the regularity of $\reg$ (such as having a Lipschitzian boundary) Theorems \ref{mainres1:thm} and \ref{mainres2:thm} can be extended to cover such operators (see Remark \ref{MaxOpsandNeu:rem} for some further details). 
However the operator defined in this manner is not always non-negative (so cannot be the square of a Dirac operator). 
Only relatively crude estimates for the asymptotics of the size and number of negative eigenvalues follow from the immediate extensions to our results; in particular, for any $\epsilon>0$ the number of eigenvalues below $-\epsilon t$ is $o(t)$ as $t\to+\infty$.
Further work would be needed to determine whether (the majority of) those eigenvalues guaranteed by Theorem \ref{mainres2:thm} have small absolute value (so can be regarded as belonging to approximate zero modes).
\end{rem}

\begin{rem}
When $\magf$ is constant the spectra of the Pauli and magnetic Schr\"{o}\-din\-ger operators on $\R^2$ reduces to a set Landau levels. For non-constant fields this level structure is destroyed, with the typical exception of the zero energy level of the Pauli operator (the Aharonov-Casher zero modes). On the other hand the lower and upper bounds given by Theorem \ref{mainres1:thm} will differ for any $\Lambda$ corresponding to a Landau level generated by a value at which $\magf$ is locally constant on some region (these are precisely the $\Lambda$ for which the set in \eqref{contlandaulevel:eq} has non-zero measure). It is likely that a version of Theorem \ref{mainres2:thm} could be extended to such cases. A related problem of eigenvalue accumulation near Landau levels after perturbation by a decaying electric potential has been considered; see \cite{DR} and references therein.
\end{rem}

\medskip

Precise definitions and various preliminary results are collected in Section \ref{OpBackGnd:sec}; in particular, the Lichnerowicz formula \eqref{PauliMagSch:eq} is justified (Proposition \ref{PauliMagSch:prop}), a priori bounds on $\gcf{\reg,\magp}{\lambda}$ are obtained (Proposition \ref{upbndPauli:prop}) and gauge transformations are discussed (Proposition \ref{PauliUe:prop}). 

The proof of Theorem \ref{mainres1:thm} is given in Section \ref{genasymp:sec}. This follows a standard localisation type argument (c.f., \cite{CdV}, \cite{Sob2}) using a sequence of piecewise constant approximations to $\magf$ based on increasingly fine tilings of $\rega$ by squares (Section~\ref{local:sec}). The corresponding approximation results for quadratic forms are obtained in Section~\ref{quadformest:sec} while the necessary eigenvalue counting function results for constant fields on a square are given in Section~\ref{constBsq:sec} (these are taken almost directly from \cite{CdV}). The bounds in Theorem \ref{mainres1:thm} are finally pieced together in Section~\ref{Thm1Pf:sec}.

Theorem \ref{mainres2:thm} is justified in Section \ref{AZM:sec} by initially reducing the problem to the case of fields of constant sign on a disc (Section \ref{reddisc:sec}). Suitable test functions on the disc can be constructed from holomorphic functions with the help of the ``real gauge'' transformation introduced in \cite{AC}; these functions need to be cut-off at the boundary, a process which ultimately leads to a spectral problem on the circle (Section \ref{redcir:sec}, with further technical details in Sections \ref{RedBnd:sec} and \ref{AnayCir:sec}).

\subsection*{Notation}

For a bounded open region $\reg\subset\R^2$ we use $\cont(\reg)$, $\cont[\alpha](\reg)$ and $\hol[\reg]$ to denote the space of continuous, H\"{o}lder continuous and holomorphic functions on $\reg$, without restriction on behaviour near the boundary; we replace $\reg$ with $\overline{\reg}$ to indicate uniform versions of the same spaces. For $k\in\NZ$ we use $\cont[k,\alpha](\reg)$ and $\sob[]{k,2}(\reg)$ to denote the H\"{o}lder-Zygmund and Sobolev space consisting of functions with $k$ derivatives in $\cont[\alpha](\reg)$ and $L^2(\reg)$ respectively. The completion of $\cont[\infty]_0(\reg)$ in $\sob[]{k,2}(\reg)$ is denoted by $\sob{k,2}(\reg)$, while $\cont[k]_0(\reg)$ denotes the space of $k$-times continuously differentiable functions with compact support contained in $\reg$.
Unless otherwise indicated norms and inner-products are defined in the relevant $L^2$ sense.

The open disc with radius $R>0$ and centre $a\in\R^2$ is denoted $\disc[R](a)$. 
When $a=0$ or $R=1$ these values are omitted; in particular, $\disc$ is the open unit disc.
We also set $(x)_+=\max\{x,0\}$, the positive part of $x\in\R$.

General positive constants are denoted by $C$, with numerical subscripts used to keep track of particular constants in subsequent discussions.

\section{Preliminaries}
\label{OpBackGnd:sec}

Let $\magp\in L^2_\loc(\reg,\R^2)$. Consider the Dirac operator, initially defined by 
\begin{equation}
\label{Diracdef:eq}
\Dirac{\magp}u=\sigma.\mm{\magp}u=\begin{pmatrix}0&\mm{\magp,-}\\\mm{\magp,+}&0\end{pmatrix}\begin{pmatrix}u_+\\u_-\end{pmatrix}
\end{equation}
for $u\in\cont[\infty]_0(\reg,\C^2)$. The operator $\Dirac{\magp}$ is densely defined and symmetric, hence closable; by a slight abuse of notation we will also denote the closure by $\Dirac{\magp}$. 

\begin{rem}
Alternatively we can proceed by considering the operators $\mm{\magp,\pm}$ separately. Initially densely defined on $\cont[\infty]_0(\reg)$ these operators satisfy $\mm{\magp,\pm}\subseteq\mm[\,*]{\magp,\mp}$ and are hence closable. Using the same notation for the closures \eqref{Diracdef:eq} then holds for all $u\in\Dom(\Dirac{\magp})=\Dom(\mm{\magp,+})\times\Dom(\mm{\magp,-})$.
\end{rem}

Define a quadratic form by
\[
\mathbf{p}_{\reg,\magp}(u)=\norm{\Dirac{\magp}u}^2
=\norm{\mm{\magp,+}u_+}^2+\norm{\mm{\magp,-}u_-}^2,
\quad u\in\Dom(\Dirac{\magp}).
\]
Since $\Dirac{\magp}$ is a closed operator $\mathbf{p}_{\reg,\magp}$ is a closed non-negative quadratic form. The \emph{Pauli operator} on $\reg$ with magnetic potential $A$ and Dirichlet boundary conditions is defined to be the corresponding self-adjoint operator given by the representation theorem; we'll use the notation $\Pauli{\reg,A}$.

\begin{rem}[Case $\magp\equiv0$]
\label{A0Pauli:rem}
Since $\reg$ is bounded $(\norm{\nabla_1u}^2+\norm{\nabla_2u}^2)^{1/2}$ gives an equivalent norm on the Sobolev space $\sob{1,2}(\reg)$ (see \cite{A}). 
Also, for $u\in\cont[\infty]_0(\reg)$, 
\[
\norm{\mm{0,\pm}u}^2=\norm{\nabla_1u}^2+\norm{\nabla_2u}^2\mp i\int_\reg\bigl(\nabla_1\overline{u}\,\nabla_2u-\nabla_2\overline{u}\,\nabla_1u\bigr)
=\norm{\nabla_1u}^2+\norm{\nabla_2u}^2.
\]
Completion then gives $\Dom(\Dirac{0})=\sob{1,2}(\reg,\C^2)$ 
with $\mathbf{p}_{\reg,0}(u)\asymp\norm{u}_{\sob{1,2}(\reg,\C^2)}^2$.
\end{rem}

\begin{rem}
\label{PauliA+-symm:rem}
Since $\mm{\magp,\pm}\ov{u}=-\ov{\mm{-A,\mp}u}$ when $u\in\cont[\infty]_0(\reg)$ we get $\mathbf{p}_{\reg,\magp}(\mathcal{J}u)=\mathbf{p}_{\reg,-\magp}(u)$ for all $u\in\Dom(\mathbf{p}_{\reg,\magp})=\Dom(\mathbf{p}_{\reg,-\magp})$, where $\mathcal{J}$ is the anti-linear isometric involution on $L^2(\reg,\C^2)$ defined by 
\[
\mathcal{J}u=\begin{pmatrix}\ov{u_-}\\\ov{u_+}\end{pmatrix},
\quad u=\begin{pmatrix}u_+\\u_-\end{pmatrix}\in L^2(\reg,\C^2).
\]
It follows that $\mathcal{J}\Pauli{\reg,\magp}\mathcal{J}=\Pauli{\reg,-\magp}$, 
and so $\Pauli{\reg,\magp}$ and $\Pauli{\reg,-\magp}$ have the same spectrum.
\end{rem}

\medskip

In order to make use of results for Schr\"{o}dinger operators we will need a rigorous form of the Lichnerowicz formula \eqref{PauliMagSch:eq}.
We begin by introducing the magnetic Schr\"{o}dinger operator in a way that parallels our introduction of the Pauli operator. 

For $l=1,2$ we initially define the operator $\mm{\magp,l}=-i\nabla_l-A_l$ on $\cont[\infty]_0(\reg)$. This operator is densely defined and symmetric, hence closable; by a slight abuse of notation we will also denote the closure by $\mm{\magp,l}$. Setting
\[
\msob(\reg)=\Dom(\mm{\magp,1})\cap\Dom(\mm{\magp,2}),
\]
the quadratic form defined by 
\[
\mathbf{h}_{\reg,\magp}(u)=\norm{\mm{\magp,1}u}^2+\norm{\mm{\magp,2}u}^2,\quad u\in\msob(\reg)
\]
is closed and non-negative. The \emph{magnetic Schr\"{o}dinger operator} on $\reg$ with magnetic potential $\magp$ and Dirichlet boundary conditions is defined to be the corresponding self-adjoint operator given by the representation theorem; we'll use the notation $\Sch{\reg,\magp}$.

\begin{rem}[Case $\magp\equiv0$]
\label{A0MagSch:rem}
It is straightforward to see that $\msob[0](\reg)=\sob{1,2}(\reg)$ with $\mathbf{h}_{\reg,0}(u)\asymp\norm{u}_{\sob{1,2}(\reg)}^2$ (c.f., Remark \ref{A0Pauli:rem}). Furthermore $\Sch{\reg,0}=-\Delta_{\reg}$, the laplacian on $\reg$ with Dirichlet boundary conditions.
\end{rem}

We need to add a scalar potential to the operator $\Sch{\reg,\magp}$. Since $\Sch{\reg,\magp}$ is a semi-bounded self-adjoint operator this can be done conveniently via the standard KLMN construction if the scalar potential is relatively form bounded with respect to $\Sch{\reg,\magp}$ with relative bound less than $1$ (see \cite{RS2}, for example). 
If $V\in L^1_\loc(\reg)$ is a real-valued function then the form given by 
\begin{equation}
\label{Vform:eq}
\mathbf{v}(u)=\ipd{u}{Vu}
\end{equation}
is certainly defined for $u\in\cont[\infty]_0(\reg)$. To extend $\mathbf{v}$ to $\msob[0](\reg)=\sob{1,2}(\reg)$ we need to restrict $V$ to the Orlicz space $L\log L(\reg)$. More precisely, introduce the $N$-function 
\[
\LA(t)=(t+1)\log(t+1)-t,\quad t\ge0; 
\]
we then define $L\log L(\reg)$ to be the Orlicz space $L_{\LA}(\reg)$ (see \cite{A}). It is straightforward to check that $L^p(\reg)\subset L_{\LA}(\reg)\subset L^1(\reg)$ for any $p>1$. Now suppose $V\in L_{\LA}(\reg)$. By \cite[Lemma 2.1]{Sol} (see also Remark \ref{comSolres:rem} below) $\mathbf{v}$ given by \eqref{Vform:eq} then defines a bounded form on $\sob{1,2}(\reg)$, while the corresponding operator $T_V$ is compact. Viewing $T_V$ as multiplication by $V$ acting as a mapping $\sob{1,2}(\reg)\to(\sob{1,2}(\reg))^*$, it follows that $V$ is relatively form compact and hence infinitesimally form bounded with respect to $\Sch{\reg,0}$. The KLMN construction can then be used to define $\Sch{\reg,0}-V$. Since the Dirichlet laplacian $-\Delta_\reg=\Sch{\reg,0}$ has a compact resolvent (see \cite{RS4}, for example) the infinitesimal form boundedness of $V$ implies $\Sch{\reg,0}-V$ also has a compact resolvent (see \cite{K}).

\begin{rem}
\label{comSolres:rem}
The results we need from \cite{Sol} are mostly stated from the case of Neumann boundary conditions under the assumption that $\reg$ has a Lipschitzian boundary. However it is easy to see that they also hold in the Dirichlet case for arbitrary bounded $\reg$.
\end{rem}

The results of the above discussion can be generalised to include a magnetic potential $\magp$ with the help of the diamagnetic inequality; a convenient form of the latter can be found in \cite{HS}.

\begin{prop}
\label{magSchVbasic:prop}
Let $\magp\in L^2_\loc(\reg)$ and $V\in L_{\LA}(\reg)$. 
Then (multiplication by) $V$ is an infinitesimally form bounded perturbation of $\Sch{\reg,\magp}$. Furthermore the semi-bounded self-adjoint operator $\Sch{\reg,\magp}-V$ (resulting from the KLMN construction) has a compact resolvent.
\end{prop}

\begin{proof}
The discussion proceeding the result covers the case $\magp\equiv0$. Using \cite[Theorem 3.3; see also Remark 3.4(i)]{HS} it follows that $V$ is infinitesimally form bounded with respect to $\Sch{\reg,\magp}$, while, for $t>0$, 
\begin{equation}
\label{diamagineq:eq}
e^{-t(\Sch{\reg,\magp}-V)}\preccurlyeq e^{-t(\Sch{\reg,0}-V)}
\end{equation}
(where $S\preccurlyeq T$ means that $S$ is dominated by $T$). However if $S\preccurlyeq T$ and $T$ is compact then $S$ must also be compact (see \cite{DF}, \cite{P}), while for a semi-bounded self-adjoint operator $Q$, $e^{-Q}$ is compact iff $Q$ has a compact resolvent. It follows that $\Sch{\reg,\magp}-V$ has a compact resolvent.
\end{proof}

Our a priori bounds for the counting function of the Pauli operator can be obtained from suitable bounds on the number of negative eigenvalues of $\Sch{\reg,\magp}-V$. 
The latter will be obtained through a two step process; results from \cite{Sol} allow us to estimate the counting function for $\Sch{\reg,0}-V$ under the condition that $V\in L_{\LA}(\reg)$, while the techniques of \cite{R} allow us to use the diamagnetic inequality (see \eqref{diamagineq:eq}) to generalise to $\Sch{\reg,\magp}-V$.

\begin{prop}
\label{upbndSchcnt:prop}
Let $\magp\in L^2_\loc(\reg)$ and $V\in L_{\LA}(\reg)$. Then
\[
\#\bigl\{\lambda_n(\Sch{\reg,\magp}-V)\le0\bigr\}\le\Co{1}\norm{V}_{L_{\LA}(\reg)}.
\]
\end{prop}

\begin{proof}
Since the positive and negative parts of any $V\in L_{\LA}(\reg)$ also belong to $L_{\LA}(\reg)$, while the addition of a positive scalar potential can only raise eigenvalues, it suffices to prove the result assuming $V\ge0$.

Now $(\Sch{\reg,0})^{1/2}:\msob[0](\reg)\to L^2(\reg)$ is an isomorphism (this is equivalent to the fact that $\mathbf{h}_{\reg,0}(u)\asymp\norm{u}_{\sob{1,2}(\reg)}^2$ on $\msob[0](\reg)=\sob{1,2}(\reg)$; see Remark \ref{A0MagSch:rem}).
Thus the expression $S_V=((\Sch{\reg,0}{})^{-1/2})^*\,T_V(\Sch{\reg,0})^{-1/2}$ defines a non-negative self-adjoint operator on $L^2(\reg)$.
By \cite[Corollary 2.3]{Sol} $T_V$ and hence $S_V$ belong to the weak first Schatten class with $\norm{S_V}_{1,w}\le\Co{1,1}\norm{V}_{L_{\LA(\reg)}}$ for some constant $\Co{1,1}$; in other words
\[
0\le\lambda_n(S_V)\le\Co{1,1}\norm{V}_{L_{\LA(\reg)}}\,n^{-1},\quad n\in\N.
\]
A standard Birman-Schwinger type argument then gives
\[
\#\bigl\{\lambda_n(\Sch{\reg,0}-\gamma V)\le0\bigr\}\le\Co{1,1}\gamma\norm{V}_{L_{\LA}(\reg)},\quad\gamma\ge0.
\]
Denote the right hand side as $\mu(\gamma)$ and let $\widehat{\mu}$ be the Laplace transform of $\mu$; in particular $\gamma^{-1}\widehat{\mu}(\gamma^{-1})=\mu(\gamma)$. Now \eqref{diamagineq:eq} (with $V\equiv0$) gives $e^{-t\Sch{\reg,\magp}}\preccurlyeq e^{-t\Sch{\reg,0}}$. Using \cite[Theorem 3]{R} we then obtain
\[
\#\bigl\{\lambda_n(\Sch{\reg,\magp}-\gamma V)\le0\bigr\}\le e\gamma^{-1}\widehat{\mu}(\gamma^{-1})
=\Co{1}\gamma\norm{V}_{L_{\LA}(\reg)},\quad\gamma\ge0,
\]
where $\Co{1}=e\Co{1,1}$.
\end{proof}

\medskip

We can now compare the Pauli operator with the magnetic Schr\"{o}dinger operator. We begin by looking at the corresponding forms. 
For any $\magp\in L^2_\loc(\reg)$ and $u\in\cont[\infty]_0(\reg)$ we can define $\mathbf{b}(u)$ to be the distribution $\magf=\nabla_1\magp_2-\nabla_2\magp_1$ acting on the test function $\abs{u}^2\in\cont[\infty]_0(\reg)$. If $\magf\in L^1_\loc(\reg)$ then
\[
\mathbf{b}(u)=\int_\reg \magf\abs{u}^2=\ipd{u}{\magf u}; 
\]
that is, $\mathbf{b}$ is just the form associated with the operator of multiplication by $\magf$.

\begin{lem}
\label{PauliMagSchform:lem}
Let $\magp\in L^2_\loc(\reg)$. 
If $u\in\cont[\infty]_0(\reg,\C^2)$ then 
\begin{equation}
\label{preformPauli=Sch:eq}
\mathbf{p}_{\reg,\magp}(u)=\mathbf{h}_{\reg,\magp}(u_+)-\mathbf{b}(u_+)+\mathbf{h}_{\reg,\magp}(u_-)+\mathbf{b}(u_-).
\end{equation}
\end{lem}

\begin{proof} 
If $v\in\cont[\infty]_0(\reg)$ then $\norm{\mm{\magp,\pm}v}^2=\norm{\mm{\magp,1}v}^2+\norm{\mm{\magp,2}v}^2\mp2\im\ipd{\mm{\magp,1}v}{\mm{\magp,2}v}$ while
\begin{align*}
2\im\ipd{\mm{\magp,1}v}{\mm{\magp,2}v}
&=\int_\reg\Bigl[-i\,\nabla_1\ov{v}\,\nabla_2v-\nabla_1\ov{v}\,A_2v+A_1\ov{v}\,\nabla_2v-iA_1A_2\abs{v}^2\Bigr.\\[-4pt]
&\quad\qquad\qquad\Bigl.{}+i\,\nabla_2\ov{v}\,\nabla_1v-A_2\ov{v}\,\nabla_1v+\nabla_2\ov{v}\,A_1v+iA_2A_1\abs{v}^2\Bigr]\\
&=\int_{\reg}\bigl[-\magp_2\,\nabla_1\abs{v}^2+\magp_1\,\nabla_2\abs{v}^2\bigr].
\end{align*}
The final expression is just the distribution $\nabla_1\magp_2-\nabla_2\magp_1=\magf$ acting on $\abs{v}^2$. 
The result now follows from the definitions of $\mathbf{p}_{\reg,\magp}$, $\mathbf{h}_{\reg,\magp}$ and $\mathbf{b}$.
\end{proof}

If $\magf\in L_{\LA}(\reg)$ the operators $\Sch{\reg,\magp}\mp\magf$ can be defined as discussed above. The corresponding forms $\mathbf{h}_{\reg,\magp}\mp\mathbf{b}$ have core $\cont[\infty]_0(\reg)$, while $\cont[\infty]_0(\reg,\C^2)$ is a core for $\mathbf{p}_{\reg,\magp}$. The previous result then gives $\Dom(\mathbf{p}_{\reg,\magp})=\msob(\reg,\C^2)$, with \eqref{preformPauli=Sch:eq} extending to all $u\in\msob(\reg,\C^2)$. The operator identity \eqref{PauliMagSch:eq} now follows, allowing Proposition \ref{magSchVbasic:prop} to be applied to $\Pauli{\reg,\magp}$; we summarise what we need as follows.

\begin{prop}
\label{PauliMagSch:prop}
Suppose $\magp\in L^2_\loc(\reg)$ with $\magf\in L_{\LA}(\reg)$. Then the Lichnerowicz formula \eqref{PauliMagSch:eq} holds as an operator identity for the Pauli and magnetic Schr\"{o}dinger operators with Dirichlet boundary conditions. Furthermore $\Pauli{\reg,\magp}$ has a compact resolvent and hence discrete spectrum. 
\end{prop}

Proposition \ref{upbndSchcnt:prop} can now be used to obtain a priori bounds on $\gcf{\reg,\magp}{\lambda}$. 
However we will need uniform versions of these bounds for sub-regions of $\reg$. If $\reg'\subseteq\reg$ is open 
we can restrict $\magp$ to $\reg'$ and consider the Pauli operator $\Pauli{\reg',\magp}$ with corresponding counting function $\gcf{\reg',\magp}{\lambda}$. 
Using $\chi_{\reg'}$ to denote the characteristic function for $\reg'$ 
Proposition \ref{PauliMagSch:prop} and a simple variational argument then give
\[
\gcf{\reg',\magp}{\lambda}
\le\#\bigl\{\lambda_n(\Sch{\reg,\magp}-(\lambda+\magf)\chi_{\reg'})\le0\bigr\}
+\#\bigl\{\lambda_n(\Sch{\reg,\magp}-(\lambda-\magf)\chi_{\reg'})\le0\bigr\}, 
\]
for any $\lambda\in\R$.
However $\norm{(\lambda\pm\magf)\chi_{\reg'}}_{L_{\LA}(\reg)}\le\norm{\magf}_{L_{\LA}(\reg')}+\abs{\lambda}\,\norm{1}_{L_{\LA}(\reg')}$, 
so Proposition \ref{upbndSchcnt:prop} now completes the following.

\begin{prop}
\label{upbndPauli:prop}
Suppose $\magp\in L^2_\loc(\reg)$ with $\magf\in L_{\LA}(\reg)$. Then
\[
\gcf{\reg',\magp}{\lambda}\le2\Co{1}\bigl(\norm{\magf}_{L_{\LA}(\reg')}+\abs{\lambda}\norm{1}_{L_{\LA}(\reg')}\bigr)
\]
for any open $\reg'\subseteq\reg$; the constant $\Co{1}$ may depend on $\reg$ but not on $\reg'$.
\end{prop}

\medskip

The magnetic potentials $\magp,\magp'\in L^2_\loc(\reg)$ are \emph{gauge equivalent} if $\magp'=\magp+\nabla\psi$ for some $\psi\in\sob[\loc]{1,2}(\reg)$.
It follows that $\Curl\magp'=\Curl\magp$ (as distributions), so $\magp'$ and $\magp$ generate the same magnetic field. 
The converse is not generally true; a topological condition on $\reg$ is also required. 
The following is a particular case of \cite[Lemma 1.1]{L}.

\begin{lem}
\label{Leinfelder1.1:lem}
Suppose $\reg$ is simply connected. If $\magp,\magp'\in L^2_\loc(\reg)$ satisfy $\Curl\magp'=\Curl\magp$ (as distributions) then there exists $\psi\in\sob[\loc]{1,2}(\reg)$ with $\magp'=\magp+\nabla\psi$. 
\end{lem}

Now suppose $\magp,\magp'\in L^2_\loc(\reg)$ are gauge equivalent and $V\in L_{\LA}(\reg)$ (so that the operators $\Sch{\reg,\magp}-V$ and $\Sch{\reg,\magp'}-V$ correspond to the closures of the semi-bounded forms $\mathbf{h}_{\reg,\magp}-\mathbf{v}$ and $\mathbf{h}_{\reg,\magp'}-\mathbf{v}$ on $C^\infty_0(\reg)$). 
Choosing $\psi\in\sob[\loc]{1,2}(\reg)$ with $\magp'=\magp+\nabla\psi$, the argument given for the proof of \cite[Theorem 1.2]{L} then shows the unitary operator $U_\psi$ of multiplication by $e^{i\psi}$ gives a unitary equivalence $\Sch{\reg,\magp'}-V=U_\psi(\Sch{\reg,\magp}-V)U_\psi^*$ (note that \cite[Theorem 1.2]{L} is stated for $\reg=\R^2$ and only assumes $\Curl\magp=\Curl\magp'$; however the former is only used to guarantee the existence of $\psi$, after which the proof easily adapts to cover arbitrary $\reg$).
Coupled with Proposition \ref{PauliMagSch:prop} and Lemma \ref{Leinfelder1.1:lem} we arrive at the following. 

\begin{prop}
\label{PauliUe:prop}
Suppose $\magp,\magp'\in L^2_\loc(\reg)$ satisfy $\magp'=\magp+\nabla\psi$ for some $\psi\in\sob[\loc]{1,2}(\reg)$ (which follows from the condition $\Curl\magp'=\Curl\magp$ when $\reg$ is simply connected). Also suppose $\magf\in L_{\LA}(\reg)$. 
Then $U_\psi$ gives the unitary equivalence $\Pauli{\reg,\magp'}=U_\psi\Pauli{\reg,\magp}U_\psi^*$; in particular, $\Pauli{\reg,\magp}$ and $\Pauli{\reg,\magp'}$ have the same spectrum. 
\end{prop}

\medskip

\begin{rem}[Maximal operators]
\label{MaxOpsandNeu:rem}
For $l=1,2$ the operator $\mm{\magp,l}$ is the minimal closed extension of the magnetic momentum operator initially defined on $\cont[\infty]_0(\reg)$. The corresponding maximal closed extension thus satisfies $\mm[\max]{\magp,l}=\mm[\,*]{\magp,l}$. The closed non-negative quadratic form $\mathbf{h}_{\reg,\magp}^{\max}(u)=\norm{\mm[\max]{\magp,1}u}^2+\norm{\mm[\max]{\magp,2}u}^2$ can be used to define the magnetic Schr\"{o}dinger operator on $\reg$ with magnetic potential $\magp$ and Neumann boundary conditions (see \cite{HS} for further discussion of this operator). 

Neumann versions of Propositions \ref{magSchVbasic:prop} and \ref{upbndSchcnt:prop} are possible if we assume $\reg$ has some additional regularity. In all cases it is sufficient to assume the existence of a linear extension operator which is continuous as a map $\sob[]{k,2}(\reg)\to\sob[]{k,2}(\R^2)$ for $k=0,1$ (such operators exist if $\reg$ has a Lipschitzian boundary). 
Assuming $\reg$ satisfies such a condition and $\magf\in L_{\LA}(\reg)$ we can use \eqref{PauliMagSch:eq} to define a ``Neumann'' Pauli operator from the Neumann magnetic Schr\"{o}dinger operator; denote this operator by $\Pauli[\prime]{\reg,\magp}$, with corresponding form $\mathbf{p}_{\reg,\magp}'$.
It is then possible to extend Theorems \ref{mainres1:thm} and \ref{mainres2:thm} to cover $\Pauli[\prime]{\reg,\magp}$, for the most part simply by using a combination of variational arguments (note that $\mathbf{p}_{\reg,\magp}'$ is an extension of $\mathbf{p}_{\reg,\magp}$) and straightforward modifications to the given proofs; the most notable exception is Proposition \ref{upbndPauli:prop} and its application where, to retain uniformity in $\reg'$, one is forced to consider operators with mixed boundary conditions (Neumann on $\partial\reg'\cap\partial\reg$ and Dirichlet on $\partial\reg'\setminus\partial\reg$). 
\end{rem}

\medskip

We complete this section by considering some basic properties of the auxiliary functions $\nu^\pm$. 
If $b\neq0$ then \eqref{defnu:eq} gives $\nu(b,\lambda)=0$ when $\lambda<0$, while
\[
\nu(b,\lambda)=(2m+1)\frac{\abs{b}}{2\pi}\le\frac1{2\pi}(\lambda+\abs{b})
\]
when $2m\abs{b}<\lambda<2(m+1)\abs{b}$ for some $m\in\NZ$. It follows that $\nu$ is locally bounded, so $\nu^\pm$ are well defined locally bounded functions on $\R^2$.
Furthermore
\begin{equation}
\label{nupmbound:eq}
0\le\nu^\pm(b,\lambda)\le\frac1{2\pi}(\abs{b}+\abs{\lambda}),\quad b,\lambda\in\R.
\end{equation}
In particular \eqref{nupmbound:eq} ensures the integrals appearing in Theorem \ref{mainres1:thm} are finite whenever $B\in L^1(\reg)$. 
We further note that $\nu^\pm$ are homogeneous of degree $1$ while, for any $b,\lambda\in\R$, we have $\nu^\pm(b,\lambda)=0$ if $\lambda<0$, $\nu^\pm(0,\lambda)=\lambda/(2\pi)$ if $\lambda\ge0$ ($\nu^\pm$ are actually continuous at $b=0$), $\nu^-(b,0)=0$ and $\nu^+(b,0)=\abs{b}/(2\pi)$. 
The final identity reduces the upper bound in Theorem \ref{mainres1:thm} to \eqref{mainres1lambda0:eq} when $\Lambda=0$.

\section{General Asymptotics}
\label{genasymp:sec}

\subsection{Constant field on a square}
\label{constBsq:sec}

For $R>0$ and $b\in\R$ let $\Pauli{R,b}$ denote a Pauli operator on the square $(0,R)^2$ with Dirichlet boundary conditions and corresponding to a constant magnetic field $b$. One choice for the magnetic potential is $\magp(x)=b(-x_2,x_1)/2$, while $(0,R)^2$ is simply connected so Proposition \ref{PauliUe:prop} shows that any other choice leads to a unitarily equivalent operator. Thus the eigenvalue counting function 
\[
\gcf{R,b}{\lambda}=\#\,\bigl\{\lambda_n(\Pauli{R,b})\le\lambda\bigr\}
\]
(counting with multiplicity) depends only on $R$, $b$ and $\lambda$. 
We can estimate $\gcf{R,b}{\lambda}$ using the auxiliary function introduced in \eqref{defnu:eq}.

\begin{prop}
\label{constBsq:prop}
For any $\lambda,b\in\R$ and $\rho\in(0,1)$ we have
\[
R^2(1-\rho)^2\,\nu^+\bigl(b,\lambda-\Co{2}R^{-2}\rho^{-2}\bigr)\le\gcf{R,b}{\lambda}\le R^2\,\nu^+(b,\lambda),
\]
where $\Co{2}$ can be chosen as an absolute constant.
\end{prop}

\begin{proof}
Let $\Sch{R,\abs{b}}$ denote a Dirichlet magnetic Schr\"{o}dinger operator on $(0,R)^2$ corresponding to the constant field $\abs{b}$, 
and let $\ncf{R,\abs{b}}{\lambda}=\#\{\lambda_n(\Sch{R,\abs{b}})\le\lambda\}$
denote the associated eigenvalue counting function (including multiplicity). 
Using Remark \ref{PauliA+-symm:rem} and Proposition \ref{PauliMagSch:prop} we then get
\[
\gcf{R,b}{\lambda}=\gcf{R,\abs{b}}{\lambda}
=\ncf{R,\abs{b}}{\lambda+\abs{b}}+\ncf{R,\abs{b}}{\lambda-\abs{b}}.
\]
On the other hand, \cite[Theorem 3.1]{CdV} gives an absolute constant $\Co{2}$ such that
\[
R^2(1-\rho)^2\,\mu\bigl(\abs{b},\lambda-\Co{2}R^{-2}\rho^{-2}\bigr)\le\ncf{R,\abs{b}}{\lambda}\le R^2\,\mu(\abs{b},\lambda),
\]
where $\mu(\abs{b},\lambda)=0$ for $\lambda<0$, $\mu(0,\lambda)=\lambda/(4\pi)$ for $\lambda\ge0$, and
\[
\mu(\abs{b},\lambda)=\frac{\abs{b}}{2\pi}\,\#\bigl\{m\in\NZ:(2m+1)\abs{b}\le\lambda\bigr\}
\]
when $b\neq0$. Since $\mu(\abs{b},\lambda+\abs{b})+\mu(\abs{b},\lambda-\abs{b})=\nu^+(b,\lambda)$ the result follows.
\end{proof}

\subsection{Localisation}
\label{local:sec} 

We want to approximate the field $B\in L_{\LA}(\reg)\cap\cont(\rega)$ by a sequence of fields which take constant values on squares within $\rega$. This approximation can only be made sufficiently good where $B$ is continuous (continuity is used when making the corresponding approximation to the potential; see Lemma \ref{AApprox:lem}). In turn this necessitates a degree of delicacy in the choice of the squares and the rate at which they approach the boundary of $\rega$ (see Lemma \ref{DomApprox:lem}). 

For each $\delta>0$ set 
\[
\rega[\delta]=\bigl\{x\in\rega:\cdisc{\delta}{x}\subset\rega\}.
\]
Clearly $\rega[\delta]$ is open, $\rega[\delta]\subset\subset\rega[\delta']$ whenever $0\le\delta'<\delta$ (recall that $\rega$ is bounded) and
\begin{equation}
\label{regaUregad:eq}
\rega=\bigcup_{\delta>0}\rega[\delta].
\end{equation}
Since $B\in\cont(\rega)$ it follows that $B\in\cont(\ov{\rega[\delta]})$ for any $\delta>0$.

\begin{lem}
\label{DomApprox:lem}
We can find a strictly increasing sequence $(k_l)_{l\in\NZ}$ in $\N$ and, for each $k\ge k_0$, a finite indexing set $\sqI$ and collection of disjoint open squares $\sq{j}$, $j\in\sqI$, of side length $2^{-k}$ with the following properties: setting 
\begin{equation}
\label{Usqdef:eq}
\Usq=\interior\ov{\bigcup_{j\in\sqI}\sq{j}}
\end{equation}
then, for each $l\in\NZ$ and $k_l\le k<k_{l+1}$,
\begin{itemize}
\item[(i)]
$\rega[2^{-l+1}]\subseteq\Usq\subset\subset\rega[2^{-l}]$.
\item[(ii)]
For any $x,y\in\Usq$ with $\abs{x-y}\le2^{-k-1/2}$ we have $\abs{\magf(x)-\magf(y)}\le2^{-l}$.
\end{itemize}
\end{lem}

Each of the squares $\sq{j}$ will be a translate of $(0,2^{-k})^2$. We will use $\msq{j}$ and $\csq{j}$ to denote the corresponding translates of $[0,2^{-k})^2$ and $[0,2^{-k}]^2$; thus $\csq{j}$ is just the closure of $\sq{j}$ while $\sq{j}\subset\msq{j}\subset\csq{j}$. The set $\Usq$ is essentially the union of the squares $\sq{j}$, $j\in\sqI$, together with any edges lying between two squares. More precisely
\[
\Usq=\interior\bigcup_{j\in\sqI}\csq{j}=\interior\bigcup_{j\in\sqI}\msq{j}\,;
\]
in particular, each $x\in\Usq$ belongs to $\msq{j}$ for a unique $j\in\sqI$.

\begin{proof}[Proof of Lemma \ref{DomApprox:lem}]
For each $k\in\NZ$ and $\delta>0$ set 
\[
d_k(\delta)=\sup\bigl\{\abs{\magf(x)-\magf(y)}:x,y\in\rega[\delta],\,\abs{x-y}\le2^{-k-1/2}\bigr\}.
\]
Since $\magf\in\cont(\ov{\rega[\delta]})$ we have $d_k(\delta)\to0$ as $k\to\infty$ (for fixed $\delta$). Hence we can find a strictly increasing sequence $(k_l)_{l\in\NZ}$ in $\N$ with $d_{k_l}(2^{-l})\le 2^{-l}$ and $k_l>l$ for each $l\in\NZ$.

Let $k\ge k_0$ and choose $l\in\NZ$ so that $k_l\le k<k_{l+1}$. Consider the tiling of $\R^2$ by copies of the square $[0,2^{-k})^2$ which have been translated so that the corners lie on points of the lattice $(2^{-k}\Z)^2$. Let $\msq{j}$ for $j\in\sqI$ denote the collection of squares from this tiling whose closure lies entirely within $\rega[2^{-l}]$. Set $\sq{j}=\interior(\msq{j})$ for $j\in\sqI$ and define $\Usq$ by \eqref{Usqdef:eq}.
Clearly $\Usq\subset\subset\rega[2^{-l}]$. Now suppose $x\in\rega[2^{-l+1}]$, so $\cdisc{2^{-l+1}}{x}\subset\rega$. Let $\ov{S}$ be the closure of any square from the tiling with $x\in\ov{S}$ and let $y\in\ov{S}$. Then $\abs{x-y}\le 2^{-k+1/2}<2^{-l}$ (since $k\ge k_l>l$) so $\cdisc{2^{-l}}{y}\subset\cdisc{2^{-l+1}}{x}\subset\rega$ and hence $y\in\rega[2^{-l}]$. Thus $\ov{S}\subset\rega[2^{-l}]$ and so $S\in\{\sq{j}:j\in\sqI\}$. It follows that $x\in\Usq$. Finally, if $x,y\in\Usq$ with $\abs{x-y}\le 2^{-k-1/2}$ then $x,y\in\rega[2^{-l}]$ with $\abs{x-y}\le 2^{-k_l-1/2}$ (since $k\ge k_l$), so
\[
\abs{\magf(x)-\magf(y)}\le d_{k_l}(2^{-l})\le2^{-l}
\]
(recall the defining properties of $k_l$).
\end{proof}

By Lemma \ref{DomApprox:lem}(i) and \eqref{regaUregad:eq} we get $\bigcap_{k\ge k_0}(\reg\setminus\ov{\Usq})=\bigcap_{\delta>0}(\reg\setminus\rega[\delta])=\reg\setminus\rega$. Since $\abs{\reg}<\infty$ and $\abs{\reg\setminus\rega}=0$ it follows that
\begin{equation}
\label{UsqApproxreg:eq}
\text{$\bigabs{\reg\setminus\ov{\Usq}}\to0$ as $k\to\infty$.}
\end{equation}

For $k\ge k_0$ set $\beta_k=2^{-l}$ where $l\in\NZ$ is maximal such that $k_l\le k$.
Lemma \ref{DomApprox:lem} implies $(\beta_k)_{k\ge k_0}$ is a non-increasing positive sequence with $\beta_k\to0$ as $k\to\infty$ while, for each $k\ge k_0$, 
\begin{equation}
\label{Bestbetak:eq}
\text{$\abs{\magf(x)-\magf(y)}\le\beta_k$ whenever $x,y\in\Usq$ with $\abs{x-y}\le2^{-k-1/2}$.}
\end{equation}

For $k\ge k_0$ and $j\in\sqI$ set $\cmagf=\magf(x)$ where $x$ is the centre of the square $\sq{j}$. 

\begin{lem}
\label{AApprox:lem}
For any $k\ge k_0$ and $j\in\sqI$ we can find a potential $\amagp\in L^2_\loc(\sq{j})$ with $\Curl\amagp=\cmagf$ and $\norm{\magp-\amagp}_{L^\infty(\sq{j})}\le\alpha_k$ where $\alpha_k=2^{-k-3/2}\beta_k$.
\end{lem}

In particular, the potential $\amagp$ generates the constant field $\cmagf$ on $\sq{j}$.

\begin{proof}
Let $k\ge k_0$ and $j\in\sqI$. For convenience centre $\sq{j}$ at the origin and set
\[
\wtamagp_1(x)=-\frac12\int_0^{x_2}\magf(x_1,t)\d t
\quad\text{and}\quad
\wtamagp_2(x)=\frac12\int_0^{x_1}\magf(t,x_2)\d t.
\]
Then $\wtamagp$, $\nabla_2\wtamagp_1$ and $\nabla_1\wtamagp_2$ are all continuous on $\sq{j}$ (since $\magf$ is continuous) while $\Curl\wtamagp=\magf=\Curl\magp$. By Lemma \ref{Leinfelder1.1:lem} 
we can then find $\psi^{k,j}\in\sob[\loc]{1,2}(\sq{j})$ with $A-\wtamagp=\nabla\psi^{k,j}$ on $\sq{j}$. Now set
\[
\amagp_1(x)=\nabla_1\psi^{k,j}-\frac12\cmagf x_2
\quad\text{and}\quad
\amagp_2(x)=\nabla_2\psi^{k,j}+\frac12\cmagf x_1.
\]
Then $\amagp\in L^2_\loc(\sq{j})$ with $\Curl\amagp=\cmagf=\magf(0)$. 
For $x\in\sq{j}$ \eqref{Bestbetak:eq} leads to
\[
\bigabs{\magp_1(x)-\amagp_1(x)}
=\lrabs{\frac12\int_0^{x_2}(\magf(x_1,t)-\magf(0))\d t}
\le2^{-k-2}\beta_k
\]
(note that $\abs{x_2}\le2^{-k-1}$). 
Clearly a similar estimate holds for $\magp_2-\amagp_2$.
\end{proof}

Also let $\chi_j$ denote the characteristic function of the set $\msq{j}\cap\Usq$, restricted to $\reg$. 
Define a piecewise constant field $\amagf:\reg\to\R$ by
\begin{equation}
\label{defBn:eq}
\amagf=\sum_{j\in\sqI}\cmagf\chi_j.
\end{equation}
The approximation $\amagf$ converges to $\magf$ pointwise on $\rega$ as $k\to\infty$; 
it will be helpful to combine this convergence with the Fatou-Lebesgue theorem as follows.

\begin{lem}
\label{intBkconv:lem}
Suppose $k(t)\in\N$ and $\Gamma(t)\in\R$ for each $t>0$ with $k(t)\to\infty$ and $\Gamma(t)\to\Gamma$ as $t\to\infty$. Then
\[
\liminf_{t\to\infty}\int_\reg\nu^+\bigl(\amagf[k(t)],\Gamma(t)\bigr)
\ge\int_\reg\nu^-(\magf,\Gamma)
\]
and
\[
\limsup_{t\to\infty}\int_\reg\nu^+\bigl(\amagf[k(t)],\Gamma(t)\bigr)
\le\int_\reg\nu^+(\magf,\Gamma). 
\]
\end{lem}

\begin{proof}
Let $k\ge k_0$. If $x\in\Usq$ then $x\in\msq{j}$ for some $j\in\sqI$ and so $\amagf(x)=\cmagf=B(x_0)$ where $x_0$ denotes the centre of $\sq{j}$. However $\abs{x-x_0}\le 2^{-k-1/2}$ so 
\begin{equation}
\label{BnBapprox:eq}
\abs{\amagf(x)-\magf(x)}=\abs{\magf(x_0)-\magf(x)}\le\beta_k
\end{equation}
by \eqref{Bestbetak:eq}. Since $\beta_k\le1$ it follows that $\abs{\amagf(x)}\le\abs{\magf(x)}+1$. 
This estimate is also valid when $x\notin\Usq$ since $\amagf(x)=0$ in this case.
For any $\lambda\in\R$ \eqref{nupmbound:eq} now gives
\[
0\le\nu^+(\amagf,\lambda)\le\frac1{2\pi}\bigl(\abs{B}+1+\abs{\lambda}\bigr).
\]

If $x\in\rega$ then \eqref{regaUregad:eq} and Lemma \ref{DomApprox:lem}(i) imply $x\in\Usq$ for all sufficiently large $k$, 
so $\amagf(x)\to\magf(x)$ as $k\to\infty$ by \eqref{BnBapprox:eq}. Since $\abs{\reg\setminus\rega}=0$ it follows that
$\amagf$ converges to $\magf$ pointwise almost everywhere on $\reg$ as $k\to\infty$. 
The result now follows from the Fatou-Lebesgue theorem (recall that $\nu^-\le\nu^+$ while $\nu^-$ and $\nu^+$ are lower and upper semi-continuous respectively). 
\end{proof}

\subsection{Quadratic form estimates}
\label{quadformest:sec}

Recall the notation introduced in Lemma \ref{DomApprox:lem}.
For each $k\ge k_0$ and $\delta>0$ set
\[
\dbsq=\bigcap_{j\in\sqI}\bigcup_{x\in\R^2\setminus\sq{j}}\disc[2^{-k}\delta](x)
\quad\text{and}\quad
\drsq=\dbsq\cap\reg.
\]
Thus $\drsq$ is an open subset of $\reg$ which contains all of $\reg\setminus\Usq$, together with a $2^{-k}\delta$-neighbourhood of the boundary of each square $\sq{j}$, $j\in\sqI$. In particular, any point of $\drsq\cap\ov{\Usq}$ must lie in $\csq{j}$ for some $j\in\sqI$, at a distance of less than $2^{-k}\delta$ from the boundary (of $\sq{j}$). Since $\sq{j}$ has side length $2^{-k}$ it follows that $\bigabs{\drsq\cap\ov{\Usq}}\le\abs{\sqI}2^{-2k+2}\delta$.
However $\abs{\sqI}\le2^{2k}\abs{\reg}$ (since the disjoint squares $\sq{j}$, $j\in\sqI$ are all contained in $\reg$) so
\begin{equation}
\label{remBgenLpest:eq}
\abs{\drsq}\le\bigabs{\reg\setminus\ov{\Usq}}+\bigabs{\drsq\cap\ov{\Usq}}
\le\bigabs{\reg\setminus\ov{\Usq}}+4\abs{\reg}\delta.
\end{equation}

We will need a partition of unity which is subordinate to the cover of $\R^2$ given by $\dbsq$ and $\sq{j}$ for $j\in\sqI$.
Using a standard construction we can find $\sco{0}\in\cont[\infty](\R^2)$ and $\sco{j}\in\cont[\infty]_0(\sq{j})$ for $j\in\sqI$ so that
\begin{equation}
\label{partuniprops:eq}
\sco{0}^2+\sum_{j\in\sqI}\sco{j}^2=1
\quad\text{and}\quad
\abs{\nabla\sco{0}}^2+\sum_{j\in\sqI}\abs{\nabla\sco{j}}^2
\le\Co{3}2^{2k}\delta^{-2},
\end{equation}
where the constant $\Co{3}$ can be chosen independently of $k$ and $\delta$ (note that, $\nabla\sco{j}$ is non-zero only in a $2^{-k}\delta$-neighbourhood of the boundary of $\sq{j}$). 
Also recall the approximating magnetic potential $\amagp$ introduced in Lemma \ref{AApprox:lem}.

\begin{prop}
\label{quadformapprox:prop}
Let $t>0$ and $\epsilon\in(0,1)$. Then
\begin{equation}
\label{quadupest:eq}
\mathbf{p}_{\reg,t\magp}(u)
\le(1-\epsilon)^{-1}\sum_{j\in\sqI}\mathbf{p}_{\sq{j},t\amagp}(u_j)
+\epsilon^{-1}t^2\alpha_k^2\norm{u}^2
\end{equation}
whenever $u=\sum_{j\in\sqI}u_j$ with $u_j\in\cont[\infty]_0(\sq{j},\C^2)$, $j\in\sqI$.
On the other hand,
\begin{equation}
\label{quadloest:eq}
\mathbf{p}_{\reg,t\magp}(u)
\ge\mathbf{p}_{\drsq,t\magp}(\sco{0}u)
+(1+\epsilon)^{-1}\sum_{j\in\sqI}\mathbf{p}_{\sq{j},t\amagp}(\sco{j}u)
-\bigl(\epsilon^{-1}t^2\alpha_k^2+\Co{3}2^{2k}\delta^{-2}\bigr)\norm{u}^2
\end{equation}
for any $u\in\msob[t\magp](\reg,\C^2)$.
\end{prop}

\begin{proof}
Firstly let $j\in\sqI$ and suppose $w\in\cont[\infty]_0(\sq{j},\C^2)$. Then
\[
\mathbf{p}_{\sq{j},t\magp}(w)
=\norm{\Dirac{t\magp}w}^2
=\bignorm{\mathcal D_{t\amagp}w-t\sigma.(\magp-\amagp)w}^2.
\]
Now $\abs{\sigma.(\magp-\amagp)\,\xi}=\abs{\magp-\amagp}\,\abs{\xi}$ for any $\xi\in\C^2$, so Lemma \ref{AApprox:lem} gives
\[
\norm{\sigma.(\magp-\amagp)w}^2
\le\norm{\magp-\amagp}_{L^\infty(\sq{j})}^2\norm{w}^2
\le\alpha_k^2\norm{w}^2.
\]
Basic norm estimates then lead to
\begin{align}
&(1+\epsilon)^{-1}\mathbf{p}_{\sq{j},t\amagp}(w)-\epsilon^{-1}t^2\alpha_k^2\norm{w}^2
\nonumber\\
\label{normestpsqjA:eq}
&\qquad\qquad{}\le\mathbf{p}_{\sq{j},t\magp}(w)
\le(1-\epsilon)^{-1}\mathbf{p}_{\sq{j},t\amagp}(w)+\epsilon^{-1}t^2\alpha_k^2\norm{w}^2.
\end{align}
Taking completions extends this estimate to any $w\in\msob[t\amagp](\sq{j},\C^2)$.

If $u=\sum_{j\in\sqI}u_j$ with $u_j\in\cont[\infty]_0(\sq{j},\C^2)$, $j\in\sqI$, then
\[
\norm{u}^2=\sum_{j\in\sqI}\norm{u_j}^2
\quad\text{and}\quad
\mathbf{p}_{\reg,t\magp}(u)=\sum_{j\in\sqI}\mathbf{p}_{\sq{j},t\magp}(u_j)
\]
since the $\sq{j}$'s are disjoint; \eqref{quadupest:eq} now follows from the second estimate in \eqref{normestpsqjA:eq}.

\smallskip

Now suppose $v\in\cont[\infty]_0(\reg)$. Enlarge $\sqI$ to $\sqI'$ to include an index for $\sco{0}$ and let $j\in\sqI'$. 
Since $\mm{t\magp}(\sco{j}v)=\sco{j}\mm{t\magp}v-i(\nabla\sco{j})v$ the first part of \eqref{partuniprops:eq} gives
\[
\sum_{j\in\sqI'}\abs{\mm{t\magp,l}(\sco{j}v)}^2
=\abs{\mm{t\magp,l}v}^2+\sum_{j\in\sqI'}(\nabla_l\sco{j})^2\,\abs{v}^2
\]
for $l=1,2$. Integration over $\reg$ and the second part of \eqref{partuniprops:eq} then lead to
\[
\mathbf{h}_{\reg,t\magp}(v)
\ge\sum_{j\in\sqI'}\mathbf{h}_{\reg,t\magp}(\sco{j}v)-\Co{3}2^{2k}\delta^{-2}\norm{v}^2.
\]
An easy calculation also gives $\mathbf{b}(v)=\sum_{j\in\sqI'}\mathbf{b}(\sco{j}v)$. Hence
\begin{equation}
\label{Pauliquadsqest:eq}
\mathbf{p}_{\reg,t\magp}(u)
\ge\mathbf{p}_{\drsq,t\magp}(\sco{0}u)+\sum_{j\in\sqI}\mathbf{p}_{\sq{j},t\magp}(\sco{j}u)-\Co{3}2^{2k}\delta^{-2}\norm{u}^2
\end{equation}
for any $u\in\cont[\infty]_0(\reg,\C^2)$. For such $u$ \eqref{quadloest:eq} now follows from the first estimate in \eqref{normestpsqjA:eq} and the fact that $\sum_{j\in\sqI}\norm{\sco{j}u}^2=\norm{u}^2-\norm{\sco{0}u}^2\le\norm{u}^2$. Taking completions then gives \eqref{quadloest:eq} for all $u\in\msob[t\magp](\reg,\C^2)$.
\end{proof}

Standard variational arguments allow us to use the quadratic form estimates of Proposition \ref{quadformapprox:prop} to obtain corresponding bounds on eigenvalue counting functions.

\begin{cor}
\label{cntfnapprox:cor}
Let $t>0$ and $\epsilon\in(0,1)$. For any $\lambda$ we have
\[
\gcf{\reg,t\magp}{\lambda}
\ge\sum_{j\in\sqI}\biggcf{\sq{j},t\amagp}{(1-\epsilon)
\bigl(\lambda-\epsilon^{-1}t^2\alpha_k^2\bigr)}
\]
and
\begin{align*}
\gcf{\reg,t\magp}{\lambda}
\le{}&\biggcf{\drsq,t\magp}{\lambda+\epsilon^{-1}t^2\alpha_k^2+\Co{3}2^{2k}\delta^{-2}}\\
&\qquad{}+\sum_{j\in\sqI}\biggcf{\sq{j},t\amagp}{(1+\epsilon)\bigl(\lambda+\epsilon^{-1}t^2\alpha_k^2+\Co{3}2^{2k}\delta^{-2}\bigr)}.
\end{align*}
\end{cor}

\subsection{Proof of Theorem \ref{mainres1:thm}}
\label{Thm1Pf:sec}

Write $\lambda(t)=(\Lambda+\gamma(t))t$ with $\Lambda\in\R$ and $\gamma(t)\to0$ as $t\to+\infty$. 
For $t>0$, $k\ge k_0$ and $\epsilon,\delta,\rho\in(0,1)$
we can combine Corollary \ref{cntfnapprox:cor}, Proposition \ref{constBsq:prop}, the homogeneity of $\nu^+$ and \eqref{defBn:eq} to get
\begin{align}
\frac1t\gcf{\reg,t\magp}{\lambda(t)}
&\ge\sum_{j\in\sqI}2^{-2k}(1-\rho)^2\,\nu^+\bigl(\cmagf,(1-\epsilon)\lbeval(t)\bigr)\nonumber\\
\label{cntupest:eq}
&=(1-\rho)^2\int_\reg\nu^+\bigl(\amagf,(1-\epsilon)\lbeval(t)\bigr) 
\end{align}
with $\lbeval(t)=\Lambda+\gamma(t)-\epsilon^{-1}t\alpha_k^2-\Co{2}(1-\epsilon)^{-1}t^{-1}2^{2k}\rho^{-2}$. Similarly
\begin{equation}
\label{cntloest:eq}
\frac1t\gcf{\reg,t\magp}{\lambda(t)}
\le\frac1t\biggcf{\drsq,t\magp}{t\ubeval(t)}+\int_\reg\nu^+\bigl(\amagf,(1+\epsilon)\ubeval(t)\bigr) 
\end{equation}
with $\ubeval(t)=\Lambda+\gamma(t)+\epsilon^{-1}t\alpha_k^2+\Co{3}t^{-1}2^{2k}\delta^{-2}$.

\smallskip

Next recall that $(\beta_k)_{k\ge k_0}$ is a non-increasing positive sequence with $\beta_k\to0$ as $k\to\infty$. Thus $(2^{2k}\beta_k^{-1})_{k\ge k_0}$ is an unbounded increasing sequence; it follows that we can define an unbounded non-decreasing function by setting 
\[
k(t)=\min\bigl\{k\ge k_0:2^{2k}\beta_k^{-1}\ge t\bigr\}
\]
for any $t>0$. Note that, if $t>2^{2k_1}\beta_{k_1}^{-1}$ then $t\in\bigl(2^{2k(t)-2}\beta_{k(t)-1}^{-1},\,2^{2k(t)}\beta_{k(t)}^{-1}\bigr]$ so
\[
t\alpha_{k(t)}^2=t2^{-2k(t)-3}\beta_{k(t)}^2\le2^{-3}\beta_{k(t)}
\quad\text{and}\quad
t^{-1}2^{2k(t)}\le 4\beta_{k(t)-1}.
\] 
Hence $t\alpha_{k(t)}^2,\,t^{-1}2^{2k(t)}\to0$ as $t\to\infty$. 
It follows that, for fixed $\epsilon$, $\delta$ and $\rho$, 
\begin{equation}
\label{lubevallim:eq}
\text{$\lbeval[k(t)](t),\,\ubeval[k(t)](t)\to\Lambda$ as $t\to\infty$.}
\end{equation}

Now put $k=k(t)$ in \eqref{cntupest:eq}. Lemma \ref{intBkconv:lem} and \eqref{lubevallim:eq} then give
\[
\liminf_{t\to\infty}\frac1t\gcf{\reg,t\magp}{\lambda(t)}
\ge(1-\rho)^2\int_{\reg}\nu^-\bigl(\magf,(1-\epsilon)\Lambda\bigr) 
\]
for any $\epsilon,\rho\in(0,1)$. Taking $\epsilon,\rho\to0^+$ (together with the Fatou-Lebesgue theorem and lower semi-continuity of $\nu^-$) now leads to the lower bound in Theorem \ref{mainres1:thm}.

\smallskip

To obtain the upper bound firstly apply Proposition \ref{upbndPauli:prop} to get the bound
\begin{equation}
\label{bnd1ub:eq}
\frac1t\biggcf{\drsq[k(t)],t\magp}{t\ubeval[k(t)](t)}\le2\Co{1}\bigl(\norm{\magf}_{L_{\LA}(\drsq[k(t)])}
+\bigabs{\ubeval[k(t)](t)}\norm{1}_{L_{\LA}(\drsq[k(t)])}\bigr).
\end{equation}
Now $B,1\in L_{\LA}(\reg)$ while \eqref{remBgenLpest:eq} and \eqref{UsqApproxreg:eq} give
$\bigabs{\drsq[k(t)]}\le\bigabs{\reg\setminus\ov{\Usq[k(t)]}}+4\abs{\reg}\delta\to4\abs{\reg}\delta$ as $t\to\infty$. 
Thus the right hand side of \eqref{bnd1ub:eq} must decay to $0$ 
if we take $t\to\infty$ and then $\delta\to0^+$. 
On the other hand Lemma \ref{intBkconv:lem} and \eqref{lubevallim:eq} give
\[
\limsup_{t\to\infty}\int_\reg\nu^+\bigl(\amagf[k(t)],(1+\epsilon)\ubeval[k(t)](t)\bigr)
\le\int_{\reg}\nu^+\bigl(\magf,(1+\epsilon)\Lambda\bigr) 
\]
for any $\epsilon\in(0,1)$. 
The upper bound in Theorem \ref{mainres1:thm} now follows it we put $k=k(t)$ in \eqref{cntloest:eq}, take $t\to\infty$ and then take $\delta,\epsilon\to0^+$.

\section{Approximate Zero Modes}
\label{AZM:sec}

\subsection{Reduction to the disc}
\label{reddisc:sec}

Most of the work in establishing Theorem \ref{mainres2:thm} lies in establishing a version of this result for single signed fields on $\disc$, the (open) unit disc in $\R^2$. We firstly quote this as a separate result and then show how the more general result follows.

\begin{thm}
\label{mainresdisc:thm}
Suppose $\magf\in\cont[\alpha](\ov{\disc})$ for some $\alpha\in(0,1)$ and $\magf$ is single signed on $\disc$.
If $\lambda(t)\ge Ce^{-ct^\sep}$ for some constants $\sep\in(0,1)$ and $c,C>0$ then
\begin{equation}
\label{mainresdisc:eq}
\liminf_{t\to\infty}\frac1t\,\gcf{\disc,t\magp}{\lambda(t)}\,\ge\,\flux[\disc]{\abs{B}}.
\end{equation}
\end{thm}

\begin{rem}[General discs]
\label{udisctoanydisc:rem}
Let $R>0$ and suppose $\magf\in\cont[\alpha](\ov{\disc}_R)$ is single signed and generated by the potential $\magp\in L^2_\loc(\disc[R])$.
Setting $\magp'(x)=R\magp(Rx)$ defines a potential $\magp'\in L^2_\loc(\disc)$ with associated field given by $\magf'(x)=R^2\magf(Rx)$; 
in particular $\magf'\in\cont[\alpha](\ov{\disc})$ is single signed and $\flux[\disc]{\abs{\magf'}}=\flux[{\disc[R]}]{\abs{\magf}}$.
On the other hand, the expression $\mathcal{U}_Ru(x)=Ru(Rx)$ defines a unitary map $\mathcal{U}_R:L^2(\disc[R])\to L^2(\disc)$ 
with $\mathcal{U}_R\Pauli{\disc[R],t\magp}\,\mathcal{U}_R^*=R^{-2}\Pauli{\disc,t\magp'}$. Thus $\gcf{\disc[R],t\magp}{\lambda}
=\gcf{\disc,t\magp'}{R^2\lambda}$ for any $\lambda$. It follows that Theorem \ref{mainresdisc:thm} generalises to cover any disc in $\R^2$ (translation is clearly not an issue).
\end{rem}

\begin{proof}[Proof of Theorem \ref{mainres2:thm}]
Set $\rega[\pm]=\bigl\{x\in\rega:\pm\magf(x)>0\bigr\}$. Then $\rega[+]\cup\rega[-]$ is open (as $\magf$ is continuous on $\rega$), 
so the Vitali covering theorem (see \cite{Jo}, for example) allows us to find a countable sequence of mutually disjoint open discs $\disc[]^{1},\,\disc[]^{2},\dots\subseteq\rega[+]\cup\rega[-]$ with $\abs{(\rega[+]\cup\rega[-])\setminus\reg_{\disc}}=0$ where $\reg_{\disc}=\bigcup_{k\in\N}\disc{}^{k}$.
Since $\magf$ is continuous and non-zero on $\rega[+]\cup\rega[-]$ it must be single signed on each $\disc[]^k$. 
A straightforward variational argument also shows $\gcf{\reg,t\magp}{\lambda}\ge\sum_{k\in\N}\gcf{\disc{}^{k},t\magp}{\lambda}$
for any $\lambda$. If $\lambda(t)\ge Ce^{-ct^\sep}$ we can then apply Theorem \ref{mainresdisc:thm} (see also Remark \ref{udisctoanydisc:rem}) 
and the superadditivity of $\liminf$ to get
\[
\liminf_{t\to\infty}\frac1t\gcf{\reg,t\magp}{\lambda(t)}
\ge\sum_{k\in\N}\flux[\disc{}^{k}]{\abs{B}}
=\flux[\reg_{\disc}]{\abs{\magf}}.
\]
However $\flux[\reg_{\disc}]{\abs{\magf}}=\flux[\reg]{\abs{\magf}}$ since $\abs{\reg\setminus\rega}=\abs{(\rega[+]\cup\rega[-])\setminus\reg_{\disc}}=0$ while $B=0$ on $\rega\setminus(\rega[+]\cup\rega[-])$.
\end{proof}

\subsection{Reduction to the circle}
\label{redcir:sec}

From Remark \ref{PauliA+-symm:rem}
it suffices to prove Theorem \ref{mainresdisc:thm} in the case that $\magf$ is non-negative. 
Clearly we can also impose the flux normalisation condition
\begin{equation}
\label{fluxcond:eq}
\flux[\disc]{\abs{B}}
=\flux[\disc]{B}
=\frac1{2\pi}\int_{\disc}\magf
=1
\end{equation}
(note that, Theorem \ref{mainresdisc:thm} holds trivially when $B\equiv0$). 
We will henceforth assume $\magf\in\cont[\alpha](\ov{\disc})$ is non-negative on $\disc$ and satisfies \eqref{fluxcond:eq}. Let
\begin{equation}
\label{defbndB:eq}
\bndB=\norm{B}_{L^\infty(\disc)}.
\end{equation}

Since $\disc$ is simply connected Proposition \ref{PauliUe:prop} gives us the freedom to choose any magnetic potential $\magp\in L^2_\loc(\disc)$ whose associated field is $\magf$. A convenient choice can be made via the ``scalar potential''.
Firstly let $\phi:\disc\to\R$ be the solution of $\Delta\smp=\magf$ on $\disc$, with $\smp=0$ on $\partial\disc=\cir$; such a solution exists, is unique and satisfies $\smp\in\cont[2,\alpha](\ov{\disc})$ (see \cite{GT}). If we set $\magp=(-\nabla_2\smp,\nabla_1\smp)\in\cont[1,\alpha](\ov{\disc},\R^2)$ then $\magp$ is a magnetic potential with associated field $\Curl\magp=\nabla_1^2\smp+\nabla_2^2\smp=\magf$.
Furthermore $\mp i\nabla_\pm\smp=\nabla_2\smp\mp i\nabla_1\smp=-(\magp_1\pm i\magp_2)$ so
\begin{equation}
\label{conjCRop:eq}
-i\,e^{\mp t\smp}\nabla_\pm\bigl(e^{\pm t\smp}\,\cdot\,\bigr)
=-i\nabla_\pm \mp it\nabla_\pm\smp
=\mm{t\magp,\pm}\,.
\end{equation}
For $u\in\cont[2]_0(\disc,\C^2)$ it follows that
\[
\mathbf{p}_{\disc,t\magp}(u)=\bignorm{e^{-t\smp}\,\nabla_+(u_+e^{t\smp})}^2
+\bignorm{e^{t\smp}\,\nabla_-(u_-e^{-t\smp})}^2.
\]
Setting $v_\pm=u_\pm e^{\pm t\smp}$ we have $u\in\cont[2]_0(\disc,\C^2)$ iff $v\in\cont[2]_0(\disc,\C^2)$, while
\begin{equation}
\label{nmprf:eq}
\norm{u}^2=\int_{\disc}\abs{v_+}^2 e^{-2t\smp}
+\int_{\disc}\abs{v_-}^2 e^{2t\smp}
\end{equation}
and
\begin{equation}
\label{qfmprf:eq}
\mathbf{p}_{\disc,t\magp}(u)
=\int_{\disc}\abs{\nabla_+v_+}^2e^{-2t\smp}+\int_{\disc}\abs{\nabla_-v_-}^2e^{2t\smp}.
\end{equation}
It is straightforward to check that $\cont[2]_0(\disc,\C^2)$ is a core for the form $\mathbf{p}_{\disc,t\magp}$.

\medskip

Using a variational argument we can establish Theorem \ref{mainresdisc:thm} by constructing sufficiently large spaces of test functions $\mathcal{X}_t\subset\cont[2]_0(\disc,\C^2)$ for which
\[
\mathbf{p}_{\disc,t\magp}(u)\le\lambda(t)\,\norm{u}^2,
\quad u\in\mathcal{X}_t.
\]
By the strong maximum principle (see \cite{GT}, for example) $\smp$ is strictly negative on $\disc$.
As $\lambda(t)\ll1$ the exponential weights in \eqref{nmprf:eq} and \eqref{qfmprf:eq} then encourage us to seek test functions with $v_-=0$ and $\nabla_+v_+=0$, at least away from the boundary $\partial\disc=\cir$. 
Identifying $\R^2$ with $\C$ in the standard way we have $\nabla_+=2\ov{\partial}$, so $\nabla_+v_+=0$ iff $v_+$ is in $\hol$ the set of holomorphic functions on $\disc$.
To get an element of $\cont[2]_0(\disc)$ we multiply by a cut-off function $\cof\in\cont[\infty]_0(\disc)$ (which should be $\R$-valued and differ from $1$ only near $\cir$).
If we take $v_+\in\hol$ and set
\begin{equation}
\label{basictfn:eq}
u=\Bigl(\begin{matrix}\cof v_+e^{-t\smp}\\0\end{matrix}\Bigr)
\in\cont[2]_0(\disc,\C^2),
\end{equation}
then $\nabla_+(\cof v_+)=v_+\,\nabla_+\cof$ and $\abs{\nabla_+\cof}^2=\abs{\nabla\cof}^2$, so \eqref{nmprf:eq} and \eqref{qfmprf:eq} become 
\begin{equation}
\label{nmtfc2:eq}
\norm{u}^2=\int_{\disc}\abs{\cof}^2\,\abs{v_+}^2 e^{-2t\smp}
\end{equation}
and
\begin{equation}
\label{qftfc2:eq}
\mathbf{p}_{\disc,t\magp}(u)=\int_{\disc}\abs{\nabla\cof}^2\,\abs{v_+}^2 e^{-2t\smp}.
\end{equation}

The remainder of our analysis will be focused near the boundary of $\disc$ (on a neighbourhood of where $\nabla\chi\neq0$). 
The information we need about $\magf$ is captured by the boundary behaviour of $\smp$.
Let $h$ denote the outward normal derivative of $\smp$ on $\partial\disc=\cir$. 
Using polar coordinates $(r,\theta)$ on $\ov{\disc}$ we have $\dnb(\theta)=\nabla_r\smp(1,\theta)$, 
while $\dnb\in\cont[1,\alpha](\cir)$ since $\phi\in\cont[2,\alpha](\ov{\disc})$.
As a consequence of the maximum principle $\dnb$ is strictly positive (see \cite[Lemma 3.4]{GT}); the quantity
\begin{equation}
\label{defkappa:eq}
\bndh=\max\bigl\{\norm{\dnb}_{L^\infty(\cir)},\norm{1/\dnb}_{L^\infty(\cir)}\bigr\}
\end{equation}
is thus finite and positive. The divergence theorem and condition \eqref{fluxcond:eq} also give
\begin{equation}
\label{Stokesconseq:eq}
\int_{\cir}\dnb
=\int_{\partial\disc}\nabla_r\smp
=\int_{\disc}\Div\nabla\smp
=\int_{\disc}\magf=2\pi.
\end{equation}

\medskip

Let $\hardy{\cir}$ denote the Hardy space on $\cir$. Each $f\in\hardy{\cir}$ is the boundary trace of a unique function $\ext f\in\hol$ ($\ext$ is just the usual identification of $\hardy{\cir}$ with the Hardy space on $\disc$). Now let $f\in\hardy{\cir}$ and set $v_+=\ext f$. Using polar coordinates on $\ov{\disc}$ define a function $\af:[0,1]\to[0,\infty)$ by
\begin{equation}
\label{defaf:eq}
\af(r)=\int_0^{2\pi}\abs{v_+(r,\theta)}^2\,e^{-2t\smp(r,\theta)}\d\theta.
\end{equation}
When $f\neq0$ we can then set
\[
\bvaf=\bigl.\nabla_r\log(r\af(r))\bigr\rvert_{r=1}=\frac{\nabla_r\af(1)}{\af(1)}+1.
\]
To create a test function we still need to fix the cut-off function $\chi$. We want this to be radial (for convenience) and decaying in a layer of width $\delta>0$ near the boundary of $\disc$. Choose a smooth non-decreasing function $\stf[]:\R\to\R$ with $\stf[]=0$ on $(-\infty,0]$, $\stf[]=1$
on $[1,\infty)$, and $\abs{\nabla\stf[]}\le1/\sqrt2$. For $\delta\in(0,1)$ and $r\in[0,1]$ set
\begin{equation}
\label{rdeltacof:eq}
\stf(r)=\stf[](\delta^{-1}(1-r)).
\end{equation}

\begin{prop}
\label{redbnd:prop}
Suppose $0\neq f\in\hardy{\cir}$ satisfies $\bvaf\le-6\bndB t\delta$ with $t>0$ and $\delta\in(0,1/3]$. 
Let $u$ be given by \eqref{basictfn:eq} where $v_+=\ext f$ and $\cof(r,\theta)=\stf(r)$. 
Then 
\[
\mathbf{p}_{\disc,t\magp}(u)\le\frac{1}{4\delta^2}\,\exp\bigl[\bvaf\delta+6\bndB t\delta^2\bigr]\,\norm{u}^2.
\]
\end{prop}

To use this estimate we need further information on the behaviour of $\bvaf$. A summary of the necessary information is contained in the next result. 

\begin{prop}
\label{bvafest:prop}
Suppose $\nu_t>0$ for $t\ge1$. Then there exist constants $\Co{4,1},\,\Co{4,2}$ and spaces $X_t\subset\hardy{\cir}$ for $t\ge1$ 
such that $\bvaf\le-\nu_t$ for all $0\neq f\in X_t$ and 
\[
\dim X_t\ge t-\Co{4,1}\nu_t-\Co{4,2}t^{(1-2\alpha)_+}.
\]
\end{prop}

Propositions \ref{redbnd:prop} and \ref{bvafest:prop} are proved at the end of Sections \ref{RedBnd:sec} and \ref{AnayCir:sec} respectively.

\begin{rem}
The test functions given by \eqref{basictfn:eq} are purely spin-up (only the component $u_+$ is non-zero). To deal with the case $\magf\le0$ directly by an argument similar to that above we would need to consider purely spin-down test functions of the form
\[
u=\Bigl(\begin{matrix}0\\\cof v_-e^{t\smp}\end{matrix}\Bigr),
\]
where $v_-\in\ahol$ (the set of anti-holomorphic functions on $\disc$). Clearly the anti-linear isometry $\mathcal{J}$ from Remark \ref{PauliA+-symm:rem} sends test functions of one type to the other.
\end{rem}

\begin{proof}[Proof of Theorem \ref{mainresdisc:thm}]
Choose $\Co{5,1}$ and $\Co{5,2}$ so that 
\begin{equation}
\label{ccc1:eq}
t^{1-\sep}e^{-ct^\sep}\le4C\Co{5,1}^2
\quad\text{and}\quad
\Co{5,1}(\Co{5,2}-6\bndB\Co{5,1})=2c
\end{equation}
for all $t\ge0$. Set $\delta_t=\Co{5,1}t^{(\sep-1)/2}$ and $\nu_t=\Co{5,2}t^{(\sep+1)/2}$. Suppose $t\ge t_0$, where $t_0=\max\{1,(3\Co{5,1})^{2/(1-\sep)}\}$; thus $t\ge1$ and $\delta_t\in(0,1/3]$. Let $X_t\subset\hardy{\cir}$ be as given by Proposition \ref{bvafest:prop} and set
\[
\mathcal{X}_t=\left\{\Bigl(\begin{matrix}\stf[\delta_t](\ext f)\,e^{-t\smp}\\0\end{matrix}\Bigr)
:f\in X_t\right\}\subset\cont[2]_0(\disc,\C^2).
\]
Since $\ext f=0$ iff $f=0$ we get
\begin{equation}
\label{findimest:eq}
\dim\mathcal{X}_t=\dim X_t\ge t-\Co{4,1}\Co{5,2}t^{(\sep+1)/2}-\Co{4,2}t^{(1-2\alpha)_+},
\end{equation}
so $\liminf_{t\to\infty}t^{-1}\dim\mathcal{X}_t=1$.

Let $0\neq f\in X_t$ and set $\displaystyle u=\Bigl(\begin{matrix}\stf[\delta_t](\ext f)e^{-t\smp}\\0\end{matrix}\Bigr)\in\mathcal{X}_t$. 
Now
\[
\nu_t\delta_t-6\bndB t\delta_t^2
=\Co{5,1}\Co{5,2}t^{\sep}-6\bndB\Co{5,1}^2t^{\sep}
=2ct^{\sep}
\quad\text{and}\quad
\frac{1}{4\delta_t^2}e^{-ct^{\sep}}=\frac{t^{1-\sep}}{4\Co{5,1}^2}\,e^{-ct^{\sep}}\le C
\]
by \eqref{ccc1:eq}, so Proposition \ref{bvafest:prop} leads to $\bvaf\le-\nu_t\le-6\bndB t\delta_t$.
Proposition \ref{redbnd:prop} then gives
\begin{equation}
\label{varprinest:eq}
\mathbf{p}_{\disc,t\magp}(u)\le\frac{1}{4\delta_t^2}\,\exp\bigl[-\nu_t\delta_t+6\bndB t\delta_t^2\bigr]\norm{u}^2
\le Ce^{-ct^\sep}\norm{u}^2.
\end{equation}
The result now follows from a variational argument.
\end{proof}

\begin{rem}
A more precise version of (the error term in) \eqref{mainresdisc:eq} can be obtained from \eqref{findimest:eq};
setting $\Co{5,3}=\max\{1,3\Co{5,1},\Co{4,1}\Co{5,2}\}$ we get
\[
\gcf{\disc,t\magp}{\lambda(t)}\ge t-\Co{5,3}t^{(\sep+1)/2}-\Co{4,2}t^{(1-2\alpha)_+}
\]
for all $t\ge0$ (note that, the right hand side is negative when $t<t_0$).
\end{rem}

\begin{rem}
With some modifications to the proof we can extend Theorem \ref{mainresdisc:thm} to obtain lower bounds for $\gcf{\disc,t\magp}{\epsilon}$ for fixed $\epsilon>0$. With $\Co{6,1}=2\sqrt{\epsilon}$ and $\Co{6,2}=6\bndB\Co{6,1}^{-1}$  
set $\delta_t=\Co{6,1}^{-1}t^{-1/2}$ and $\nu_t=(\Co{6,1}\log t+\Co{6,2})t^{1/2}$ for $t\ge0$. Suppose $t\ge t_0$, where $t_0=\max\{1,9\Co{6,1}^{-2}\}$; thus $t\ge1$ and $\delta_t\in(0,1/3]$. Let $X_t$ and $\mathcal{X}_t$ be as above. Now
\[
\nu_t\delta_t-6\bndB t\delta_t^2
=\Co{6,1}\Co{6,1}^{-1}\log t+\Co{6,2}\Co{6,1}^{-1}-6\bndB\Co{6,1}^{-2}
=\log t
\quad\text{and}\quad
\frac{1}{4\delta_t^2}=\frac{\Co{6,1}^2t}4=\epsilon t,
\]
so the middle estimate in \eqref{varprinest:eq} gives $\mathbf{p}_{\disc,t\magp}(u)\le\epsilon\norm{u}^2$ for any $u\in\mathcal{X}_t$.
On the other hand
\[
\dim\mathcal{X}_t=\dim X_t\ge t-\Co{4,1}(\Co{6,1}\log t+\Co{6,2})t^{1/2}-\Co{4,2}t^{(1-2\alpha)_+}
\]
for any $t\ge t_0$. 
Setting $\Co{6,3}=\max\{1,\Co{4,1}(\Co{6,1}+\Co{6,2}),3/\Co{6,1}\}$ we then obtain
\[
\gcf{\disc,t\magp}{\epsilon}\ge t-\Co{6,3}t^{1/2}\log(t\!+\!2)-\Co{4,2}t^{(1-2\alpha)_+}
\]
for all $t\ge0$ (note that, the right hand side is negative when $t<t_0$).

It is possible that the $\log$ is an artefact of our method and the second term in the asymptotics of $\gcf{\disc,t\magp}{\epsilon}$ should be $O(t^{1/2})$, at least for sufficiently regular $\magf$.
\end{rem}

\subsection{Quadratic form estimates}
\label{RedBnd:sec}

The aim of this section is to prove Proposition \ref{redbnd:prop}. The presentation is simplified if we switch to polar coordinates; the magnetic momentum operators are then
\[
P_r=-i\nabla_r-t\magp_r=-i\nabla_r+\frac{t}{r}\,\nabla_\theta\smp
\quad\text{and}\quad
P_\theta=-\frac{i}r\,\nabla_\theta-t\magp_\theta=-\frac{i}r\,\nabla_\theta-t\nabla_r\smp. 
\]
For functions $u,v$ defined on $\ov{\disc}$ we will use $\ipd{u}{v}_\cir$ and $\norm{u}_\cir$ to indicate the $L^2$-inner product and norm in the $\cir$ variable only; that is,
\[
\ipd{u}{v}_\cir=\int_0^{2\pi}\ov{u(r,\theta)}\,v(r,\theta)\d\theta
\quad\text{and}\quad
\norm{u}_\cir^2=\int_0^{2\pi}\abs{u(r,\theta)}^2\d\theta,
\]
which depend on $r\in[0,1]$. 

\medskip

Firstly we take a more detailed look at the function $\af$ defined in \eqref{defaf:eq}. 

\begin{lem}
\label{afdiff:lem}
Let $f\in\hardy{\cir}$ and set $u=(\ext f)\,e^{-t\smp}$. Then
\[
\nabla_r\af=2\ipd{u}{P_\theta u}_{\cir}
\quad\text{and}\quad
\nabla_r(r\nabla_r\af)=4r\norm{P_\theta u}_{\cir}^2-2tr\ipd{u}{\magf u}_{S^1}.
\]
\end{lem}

\begin{proof}
Firstly observe that for any functions $v,w$ on $\ov{\disc}$ we have 
\begin{align}
\nabla_r\ipd{v}{w}_\cir
&=\ipd{iP_rv}{w}_\cir+\ipd{v}{iP_rw}_\cir
\nonumber\\
\label{adjDth:eq}
&=\ipd{P_\theta v}{w}_\cir+\ipd{v}{P_\theta w}_\cir
-i\ipd{Qv}{w}_\cir+i\ipd{v}{Qw}_\cir,
\end{align}
where $Q=P_r+iP_\theta=e^{-i\theta}\mm{t\magp,+}$. The second expression for $Q$ and \eqref{conjCRop:eq} lead to
\begin{equation}
\label{mompanh:eq}
Qu=-ie^{-i\theta}e^{-t\smp}(\nabla_+(\ext f))=0
\end{equation}
(recall that $\ext f\in\hol$). Now $\af=\norm{u}_{\cir}^2$ so \eqref{adjDth:eq} and \eqref{mompanh:eq} give
\begin{equation}
\label{daf:eq}
\nabla_r\af=\ipd{P_\theta u}{u}_\cir+\ipd{u}{P_\theta u}_\cir
=2\ipd{u}{P_\theta u}_\cir
\end{equation}
as $P_\theta$ is symmetric with respect to $\ipd\cdot\cdot_\cir$.
Using \eqref{adjDth:eq} and \eqref{mompanh:eq} again then gives
\begin{align*}
&\nabla_r(r\nabla_r\af)\\
&\qquad=2\ipd{P_\theta u}{rP_\theta u}_\cir+2\ipd{u}{P_\theta(rP_\theta u)}_\cir
-2i\ipd{Qu}{rP_\theta u}_\cir+2i\ipd{u}{Q(rP_\theta u)}_\cir\\
&\qquad=4r\norm{P_\theta u}_\cir^2+2i\ipd{u}{[Q,rP_\theta]u}_\cir.
\end{align*}
However $[Q,rP_\theta]=[P_r,rP_\theta]=-i\nabla_r(-rt\magp_\theta)+i\nabla_\theta(-t\magp_r)=itr\magf$.
\end{proof}

The formulae for the derivatives of $\af$ given by Lemma \ref{afdiff:lem} lead to the following.

\begin{lem}
\label{afest:lem}
Let $0\le r_0\le 1$ and $f\in\hardy{\cir}$. If $0\le b\le-(\bvaf+2\bndB t(1-r_0))$ then $e^{br}r\af(r)$ is decreasing for $r\in[r_0,1]$.
\end{lem}

\begin{proof}
For any $b\in\R$ set $\lab(r)=\log(e^{br}r\af(r))=\log(\af(r))+\log r+br$. Then
\begin{align*}
\nabla_r(r\nabla_r\lab)
&=\frac{\nabla_r(r\nabla_r\af)}{\af}-r\frac{(\nabla_r\af)^2}{\af^2}+b\\
&=\frac{r}{\af^2}\bigl[4\norm{P_\theta u}^2_{S^1}\norm{u}_{\cir}^2-4\ipd{u}{P_\theta u}_{\cir}^2\bigr]-2tr\frac{\ipd{u}{Bu}_{\cir}}{\norm{u}_{\cir}^2}+b
\end{align*}
using Lemma \ref{afdiff:lem} and the fact that $\af=\norm{u}_{\cir}^2$. Now $[\cdot]\ge0$, $\ipd{u}{Bu}_{\cir}\le\bndB\norm{u}_{\cir}^2$ (recall \eqref{defbndB:eq}) and $r\le1$. Hence $\nabla_r(r\nabla_r\lab)\ge -2\bndB t+b$. 
Integrating from $r$ to $1$ and using the fact that $\nabla_r\lab(1)=\bvaf+b$ we get
\[
r\nabla_r\lab(r)\le\nabla_r\lab(1)+\int_r^1(2\bndB t-b)\d r=\bvaf+2\bndB t(1-r)+br.
\]
If $r_0\le r\le1$ and $0\le b\le-(\bvaf+2\bndB t(1-r_0))$ it follows that the right hand side is non-positive, and hence $\nabla_r\lab(r)\le0$. Thus $\exp(\lab(r))=e^{br}r\af(r)$ is decreasing for  $r\in[r_0,1]$.
\end{proof}

\begin{proof}[Proof of Proposition \ref{redbnd:prop}]
Set $r_\delta=1-\delta$ and $b=-(\bvaf+6\bndB t\delta)\ge0$, so $\stf(r)=1$ for $r\le r_\delta$ while $e^{br}r\af(r)$ is decreasing for $r\in[1-3\delta,1]$ by Lemma \ref{afest:lem}. Using \eqref{nmtfc2:eq} and \eqref{defaf:eq} we then get 
\begin{align*}
\norm{u}^2&=\int_0^1\stf(r)^2\,r\af(r)\d r
\ge\int^{r_\delta}_{1-3\delta}r\af(r)\d r\\
&\qquad{}\ge r_\delta\af(r_\delta)\int^{r_\delta}_{1-3\delta} e^{b(r_\delta-r)}\d r
=r_\delta\af(r_\delta)\,\frac1{b}\,(e^{2b\delta}-1)
\ge r_\delta\af(r_\delta)\,2\delta e^{b\delta}.
\end{align*}

To estimate $\mathbf{p}_{\disc,t\magp}(u)$ note that $\abs{\nabla\cof}^2=\abs{\nabla\stf}^2\le1/(2\delta^2)$ and $\supp(\nabla\stf)\subseteq[r_\delta,1]$, while $r\af(r)$ is decreasing for $r\in[r_\delta,1]\subset[1-3\delta,1]$ by Lemma \ref{afest:lem}. 
Thus \eqref{qftfc2:eq} and \eqref{defaf:eq} lead to
\[
\mathbf{p}_{\disc,t\magp}(u)=\int_0^1\abs{\nabla\stf(r)}^2\,r\af(r)\d r
\le\frac{1}{2\delta^2}\int_{r_\delta}^1 r\af(r)\d r
\le \frac{1}{2\delta}\,r_\delta\af(r_\delta).
\]
Combined with the previous estimate we then get $\mathbf{p}_{\disc,t\magp}(u)\le(2\delta)^{-2}e^{-b\delta}\norm{u}^2$.
\end{proof}

\subsection{Analysis on the circle}
\label{AnayCir:sec}

Let $f\in\hardy{\cir}$ and set $u=(\ext f)\,e^{-t\smp}$. Then $u(1,\theta)=f(\theta)$ so $\af(1)=\norm{f}^2$.
Also $A_\theta(1,\theta)=\nabla_r\smp(1,\theta)=\dnb(\theta)$.
Introducing the operator $T=-i\nabla-t\dnb$ on $\cir$, Lemma~\ref{afdiff:lem} now gives $\nabla_r\af(1)=2\ipd{f}{Tf}$.
If $f\neq0$ it follows that
\begin{equation}
\label{bvafT:eq}
\bvaf=2\frac{\ipd{f}{Tf}}{\norm{f}^2}+1.
\end{equation}

To construct the space $X_t$ in Proposition \ref{bvafest:prop} we need to find $f\in\hardy{\cir}$ for which $\bvaf$ is negative. 
In view of \eqref{bvafT:eq} this leads us to consider the spectral properties of the operator $T$ on $\hardy{\cir}$. 
We begin by considering $T$ as an operator on $L^2(\cir)$ where a more explicit description is possible. 
Set
\[
\idnb(\theta)=\int_0^\theta\dnb(\omega)\d\omega
\]
so $\idnb(0)=0$, $\idnb(2\pi)=2\pi$ (recall \eqref{Stokesconseq:eq}) and $\nabla\idnb=\dnb$.
Since $\dnb$ and $1/\dnb$ are both continuous and bounded away from $0$ it follows that $\idnb$ is a $\cont[1]$-diffeomorphism of $\cir$.
Thus $\mathcal{U}f=f\circ\idnb$ defines a unitary map $\mathcal{U}$ from $L^2(\cir)$ with its usual inner-product to $L^2(\cir)$ with weighted inner-product $\ipd{\cdot}{\cdot}_\dnb$ given by $\ipd{f}{g}_\dnb=\ipd{f}{hg}=\ipd{hf}{g}$.
Using $\norm{\cdot}_\dnb$ to denote the corresponding norm we have
\begin{equation}
\label{compnorm:eq}
\bndh^{-1}\norm{f}^2\le\norm{f}_\dnb^2\le\bndh\norm{f}^2
\end{equation}
for any $f\in L^2(\cir)$ (recall \eqref{defkappa:eq}). 
The image of the standard Fourier basis under $\mathcal{U}$ is $\{\Diwef{n}:n\in\Z\}$ where $\Diwef{n}=e^{in\idnb}/\sqrt{2\pi}$ for $n\in\Z$.
In particular, any $f\in L^2(\cir)$ can be written as $f=\sum_{n\in\Z}\gamma_n\Diwef{n}$ for some constants $\gamma_n$ (given by $\gamma_n=\ipd{\Diwef{n}}{f}_\dnb$), whereupon $\norm{f}_\dnb^2=\sum_{n\in\Z}\abs{\gamma_n}^2$. Since $-i\nabla\Diwef{n}=n\dnb\,\Diwef{n}$ we get
\begin{equation}
\label{Tem:eq}
T\Diwef{n}=(n-t)\,\dnb\Diwef{n}
\end{equation}
and thus
\begin{equation}
\label{PidfTf:eq}
\ipd{f}{Tf}=\sum_{n\in\Z}(n-t)\abs{\gamma_n}^2.
\end{equation}
For $M\ge0$ let $\oproj{M}$ denote the $\ipd{\cdot}{\cdot}_\dnb$-orthogonal projection onto $\ov{\Span}\{\Diwef{n}:n>M\}$. 

\begin{lem}
\label{ipdfTfest1:lem}
Let $f\in L^2(\cir)$ and $M\ge0$. Then
\[
\ipd{f}{Tf}\le(M-t)\norm{f}_\dnb^2+\ipd{\oproj{M}f}{T\oproj{M}f}.
\]
\end{lem}

\begin{proof}
Write $f=\sum_{n\in\Z}\gamma_n\Diwef{n}$ for some $\gamma_n$. 
Then $\oproj{M}f=\sum_{n>M}\gamma_n\Diwef{n}$ so 
\[
\ipd{f}{Tf}-\ipd{\oproj{M}f}{T\oproj{M}f}
=\sum_{n\le M}(n-t)\abs{\gamma_n}^2
\le(M-t)\sum_{n\in\Z}\abs{\gamma_n}^2=(M-t)\norm{f}_\dnb^2
\]
with the help of \eqref{PidfTf:eq}.
\end{proof}

We shall now move our attention to consider $T$ acting on $\hardy{\cir}$. 
Let $\hproj{+}$ denote the orthogonal projection of $L^2(\cir)$ onto $\hardy{\cir}$ and $\hproj{-}$ its complement; that is 
\[
\hproj{+}f=\frac1{\sqrt{2\pi}}\sum_{k\ge0}\fc{f}{k}\,e^{ik\theta}
\quad\text{and}\quad
\hproj{-}=I-\hproj{+},
\]
where $\fc{f}{k}$ denotes the $k$th Fourier coefficient of $f$.
A key idea in our argument is the fact that, for large $n$, $\Diwef{n}$ and $\dnb\Diwef{n}$ ``almost'' lie in the space $\hardy{\cir}$
in the sense that $\hproj{-}\Diwef{n}$ and $\hproj{-}\dnb\Diwef{n}$ become small.
This is made more precise via the quantities
\[
\sfiw{m}=\sum_{n>m}\norm{\hproj{-}\Diwef{n}}^2 
\quad\text{and}\quad
\sfiv{m}=\sum_{n>m}\norm{\hproj{-}\dnb\Diwef{n}}^2. 
\]

\begin{prop}
\label{sfcest:prop}
There exists a constant $\Co{7}$ such that
\[
\sfiw{m},\,\sfiv{m}\le\Co{7}\,(m+1)^{-2\alpha},\quad m\ge0.
\]
\end{prop}

We shall consider families of diffeomorphisms of $\cir$ which are related to $\idnb$. 
Firstly note that a positively oriented homeomorphism of $\cir$ can be viewed 
as a continuous strictly increasing map $\psi:\R\to\R$ which satisfies $\psi(\theta+2\pi)=\psi(\theta)+2\pi$. 
If $\psi$ is differentiable then $\nabla\psi:\R\to\R$ is $2\pi$-periodic and hence can be viewed as a map on $\cir$. 
It is straightforward to check that $\psi$ is a (positively oriented) $\cont[2,\alpha]$-diffeomorphism of $\cir$ 
if $\nabla\psi\in\cont[1,\alpha](\cir)$ and $\nabla\psi$ is strictly positive;
in this case $\nabla\psi$ and $\nabla\psi^{-1}$ are both uniformly bounded away from $0$, while 
$\nabla\psi^{-1}\in\cont[1,\alpha](\cir)$.

\begin{proof}[Proof of Proposition \ref{sfcest:prop}]
For each $\tau\in[0,1]$ set
$\idnb[\tau](\theta)=\tau\idnb(\theta)+(1-\tau)\theta$. Then
$\idnb[\tau](0)=0$, $\idnb[\tau](2\pi)=2\pi$ and $\nabla\idnb[\tau]=\tau\dnb+(1-\tau)$.
Thus $\nabla\idnb[\tau]\in\cont[1,\alpha](\cir)$ with $\nabla\idnb[\tau]\ge\tau\bndh^{-1}+(1-\tau)\ge\bndh^{-1}$, 
so $\nabla\idnb[\tau]$ is bounded away from $0$ uniformly in $\tau$. 
It follows that $\idnb[\tau]$ is a $\cont[2,\alpha]$-diffeomorphism of $\cir$.
Setting
\[
\ciw=\frac1{\sqrt{2\pi}}\nabla\idnb[\tau]^{-1}
=\frac1{\sqrt{2\pi}}\frac1{(\nabla\idnb[\tau])\circ\idnb[\tau]^{-1}}
\]
we get $\ciw\in\cont[1,\alpha](\cir)$, while $\norm{\ciw}_{\cont[1,\alpha](\cir)}$ can be 
bounded uniformly for $\tau\in[0,1]$. 
Using standard estimates for the Fourier coefficients of functions in $\cont[1,\alpha](\cir)$ 
(see \cite{Katz}, for example) we can then find $\Co{7,1}$ so that
\begin{equation}
\label{fcHoldest:eq}
\abs{\fc{\ciw}{n}}\le\Co{7,1}\abs{n}^{-1-\alpha},
\quad\tau\in[0,1],\,n\neq0.
\end{equation}

Now suppose $n\ge0$ and $k>0$. Set $\tau=n/(k+n)\in[0,1]$. Then
\begin{align*}
\fc{\Diwef{n}}{-k}
&=\frac1{2\pi}\int_0^{2\pi}e^{in\idnb(\theta)+ik\theta}\d\theta
=\frac1{2\pi}\int_0^{2\pi}\!\!\frac1{\nabla\idnb[\tau](\theta)}
\,e^{i(k+n)\idnb[\tau](\theta)}\nabla\idnb[\tau](\theta)\d\theta\\
&\qquad{}=\frac1{\sqrt{2\pi}}\int_0^{2\pi}\!\ciw(\omega)\,e^{i(k+n)\omega}\d\omega
=\fc{\ciw}{-(k\!+\!n)}.
\end{align*}
For any $m\ge0$ we can now combine this with \eqref{fcHoldest:eq} to get
\[
\sfiw{m}=\sum_{n>m}\sum_{k>0}\,\bigabs{\fc{\Diwef{n}}{-k}}^2
\le\Co{7,1}^2\sum_{n>m}\sum_{k>0}(k\!+\!n)^{-2-2\alpha}
\le\frac{\Co{7,1}^2}{2\alpha} (m+1)^{-2\alpha}.
\]

We can estimate $\sfiv{m}$ using a similar argument. 
In particular, we need to consider the $\cont[2,\alpha]$-diffeomorphisms of $\cir$
given by $\iidnb[\tau](\theta)=\tau\idnb^{-1}(\theta)+(1-\tau)\theta$, $\tau\in[0,1]$. Then
\[
\civ
=\frac1{\sqrt{2\pi}}\nabla\iidnb[\tau]^{-1}
\]
is uniformly bounded in $\cont[1,\alpha](\cir)$, while for any $n\ge0$ and $k>0$ it can be shown that
the $(-k)$th Fourier coefficient of $\dnb\Diwef{n}$ is just $\fc{\civ}{-(k\!+\!n)}$. 
The remainder of the argument to estimate $\sfiv{m}$ proceeds exactly as for $\sfiw{m}$.
\end{proof}

Proposition \ref{sfcest:prop} establishes that $\sfiw{0}$ is finite; 
it follows that $\sfiw{m}$ is non-increasing with $\lim_{m\to\infty}\sfiw{m}=0$.
A similar comment applies to $\sfiv{m}$. 
Also recall that $\oproj{m}$ denotes the $\ipd{\cdot}{\cdot}_\dnb$-orthogonal projection 
onto $\ov{\Span}\{\Diwef{n}:n>m\}$.

\begin{lem}
\label{nfitonf+:lem}
Suppose $\oproj{m}f=f$ for some $f\in L^2(\cir)$ and $m\ge0$.
If $\sfiw{m}\le1/(2\bndh^2)$ then $\norm{f}_\dnb^2\le2\bndh^2\norm{\hproj{+}f}^2$.
\end{lem}

\begin{proof}
Write $f=\sum_{n>m}\gamma_n\Diwef{n}$ for some $\gamma_n$.
Then
\[
\norm{\hproj{-}f}^2
\le\sum_{n>m}\abs{\gamma_n}^2\,\sum_{n>m}\norm{\hproj{-}\Diwef{n}}^2
=\sfiw{m}\norm{f}_\dnb^2.
\]
Since $f=\hproj{+}f+\hproj{-}f$ is an orthogonal decomposition \eqref{compnorm:eq} now gives
\[
\norm{f}_\dnb^2\le\bndh^2\norm{f}^2
=\bndh^2(\norm{\hproj{+}f}^2+\norm{\hproj{-}f}^2)
\le\bndh^2\norm{\hproj{+}f}^2+\frac12\norm{f}_\dnb^2.
\]
The result follows. 
\end{proof}

\begin{lem}
\label{QMf+TQMf_+:lem}
Suppose $M>0$ and $f\in\Span\{\Diwef{n}:0<n\le M\}$. Then
\[
\ipd{\oproj{M}\hproj{+}f}{T\oproj{M}\hproj{+}f}\le M\bigl[\sfiw{0}\sfiw{M}\sfiv{0}\sfiv{M}\bigr]^{1/2}\norm{f}_\dnb^2.
\]
\end{lem}

\begin{proof}
Write $f=\sum_{0<n\le M}\gamma_n\Diwef{n}$ for some $\gamma_n$. Then
\begin{equation}
\label{AMP+fexpan:eq}
\oproj{M}\hproj{+}f
=\sum_{n'>M}\ipd{\Diwef{n'}}{\hproj{+}f}_\dnb\,\Diwef{n'}
=\sum_{n'>M}\sum_{0<n\le M}\gamma_n\ipd{\Diwef{n'}}{\hproj{+}\Diwef{n}}_\dnb\,\Diwef{n'}.
\end{equation}
Noting that $\ipd{\Diwef{n'}}{\hproj{+}\Diwef{n}}_\dnb=-\ipd{\Diwef{n'}}{\hproj{-}\Diwef{n}}_\dnb$ when $n\neq n'$, 
\eqref{PidfTf:eq} and \eqref{AMP+fexpan:eq} now give
\begin{align}
\ipd{\oproj{M}\hproj{+}f}{T\oproj{M}\hproj{+}f}
&\le\sum_{n'>M}(n'-t)\,\biggl\lvert\sum_{0<n\le M}\gamma_n\ipd{\Diwef{n'}}{\hproj{-}\Diwef{n}}_\dnb\biggr\rvert^2\nonumber\\
&\le\sum_{n'>M}n'\biggl[\,\sum_{0<n\le M}\abs{\gamma_n}^2\sum_{0<n\le M}\abs{\ipd{\Diwef{n'}}{\hproj{-}\Diwef{n}}_\dnb}^2\biggr]\nonumber\\
\label{ipdQMP+fTQMP+f:eq}
&=\norm{f}_\dnb^2\sum_{0<n\le M}\sum_{n'>M}n'\,\abs{\ipd{\Diwef{n'}}{\hproj{-}\Diwef{n}}_\dnb}^2.
\end{align}
However, for any $n,n'$ we have $\ipd{\Diwef{n'}}{\hproj{-}\Diwef{n}}_\dnb=\ipd{\hproj{-}\dnb\Diwef{n'}}{\hproj{-}\Diwef{n}}$ while 
\begin{align*}
n'\,\ov{\ipd{\Diwef{n'}}{\hproj{-}\Diwef{n}}_\dnb}
=\ipd{\hproj{-}\Diwef{n}}{n'\dnb\Diwef{n'}}
&=\ipd{\hproj{-}\Diwef{n}}{-i\nabla\Diwef{n'}}\\
&=\ipd{\hproj{-}(-i\nabla)\Diwef{n}}{\Diwef{n'}}
=n\ipd{\hproj{-}\dnb\Diwef{n}}{\hproj{-}\Diwef{n'}}.
\end{align*}
(note that, $\hproj{-}$ is an $\ipd{\cdot}{\cdot}$-orthogonal projection). Hence 
\begin{align*}
\sum_{0<n\le M}\sum_{n'>M}n'\,\abs{\ipd{\Diwef{n'}}{\hproj{-}\Diwef{n}}_\dnb}^2
&=\sum_{0<n\le M}\sum_{n'>M}n\,\ipd{\hproj{-}\dnb\Diwef{n'}}{\hproj{-}\Diwef{n}}\,\ipd{\hproj{-}\dnb\Diwef{n}}{\hproj{-}\Diwef{n'}}\\
&\le M\sum_{n>0}\,\norm{\hproj{-}\Diwef{n}}\,\norm{\hproj{-}\dnb\Diwef{n}}\,\sum_{n'>M}\norm{\hproj{-}\Diwef{n'}}\,\norm{\hproj{-}\dnb\Diwef{n'}}\\
&\le M\bigl[\sfiw{0}\sfiv{0}\sfiw{M}\sfiv{M}\bigr]^{1/2}.
\end{align*}
The result now follows from \eqref{ipdQMP+fTQMP+f:eq}.
\end{proof}

\begin{prop}
\label{ipdfFfest2:prop}
Suppose $0\le m<M\le t$. Let $X=\Span\{\Diwef{n}:m<n\le M\}$ and $X^+=\hproj{+}X\subset\hardy{\cir}$. 
If $\sfiw{m}\le1/(2\bndh^2)$ then $\dim X^+=M-m$ and
\[
\ipd{f}{Tf}\le-\bndh^2\bigl(t-M-2M\bigl[\sfiw{0}\sfiw{M}\sfiv{0}\sfiv{M}\bigr]^{1/2}\bigr)\norm{f}^2,
\quad f\in X^+.
\]
\end{prop}

\begin{proof}
Let $f\in X$ and set $f^+=\hproj{+}f\in X^+$. If $f^+=0$ then $f=0$ by Lemma \ref{nfitonf+:lem}; thus $\dim X^+=\dim X=M-m$. 
On the other hand, combining \eqref{compnorm:eq} with Lemmas \ref{ipdfTfest1:lem}, \ref{nfitonf+:lem} and \ref{QMf+TQMf_+:lem} gives
\begin{align*}
\ipd{f^+}{Tf^+}&\le(M-t)\norm{f^+}_\dnb^2+\ipd{\oproj{M}f^+}{T\oproj{M}f^+}\\
&\le\bndh^2(M-t)\norm{f^+}^2+M\bigl[\sfiw{0}\sfiw{M}\sfiv{0}\sfiv{M}\bigr]^{1/2}\norm{f}_\dnb^2\\
&\le\bndh^2\bigl(M-t+2M\bigl[\sfiw{0}\sfiw{M}\sfiv{0}\sfiv{M}\bigr]^{1/2}\bigr)\norm{f^+}^2,
\end{align*}
as required.
\end{proof}

\begin{proof}[Proof of Proposition \ref{bvafest:prop}]
Choose $m\ge0$ so that $\sfiw{m}\le1/(2\bndh^2)$ (which is possible by Proposition \ref{sfcest:prop}).
Also let 
\[
\nu_{1,t}=\frac{1}{2\bndh^2}(\nu_t+1)+2\Co{7}^2\,t^{(1-2\alpha)_+}
\]
and $M_t=\min\{n\in\NZ:n\ge t-\nu_{1,t}-1\}$; in particular, $M_t\ge t-\nu_{1,t}-1$. 
Set 
\[
X_t=\hproj{+}\Span\{\Diwef{n}:m<n\le M_t\}\subset\hardy{\cir}.
\]
Proposition \ref{ipdfFfest2:prop} gives $\dim X_t\ge M_t-m\ge t-(\nu_{1,t}+m+1)$ (note that, $X_t=\{0\}$ when $M_t\le m$). 
The required estimate for $\dim X_t$ now follows if we take $\Co{4,1}=1/(2\bndh^2)$ and $\Co{4,2}=\Co{4,1}+2\Co{7}^2+m+1$ (note that, $t^{(1-2\alpha)_+}\ge1$ for $t\ge1$).

Now let $0\neq f\in X_t$. Then $1\le M_t\le t-\nu_{1,t}\le t$ (otherwise $X_t=\{0\}$), leading to
$M_t(M_t+1)^{-2\alpha}\le M_t^{1-2\alpha}\le t^{(1-2\alpha)_+}$.
Propositions \ref{ipdfFfest2:prop} and \ref{sfcest:prop} then give
\begin{align*}
\ipd{f}{Tf}&\le\bndh^2\bigl(M_t-t+2M_t\bigl[\sfiw{0}\sfiw{M_t}\sfiv{0}\sfiv{M_t}\bigr]^{1/2}\bigr)\norm{f}^2\\
&\le\bndh^2\bigl(M_t-t+2\Co{7}^2\,M_t(M_t+1)^{-2\alpha}\bigr)\norm{f}^2\\
&\le\bndh^2\bigl(-\nu_{1,t}+2\Co{7}^2\,t^{(1-2\alpha)_+}\bigr)\norm{f}^2
=-\frac12(\nu_t\!+\!1)\,\norm{f}^2,
\end{align*}
so $\bvaf\le-\nu_t$ by \eqref{bvafT:eq}.
\end{proof}

\subsubsection*{Acknowledgements}
The author wishes to thank A.\ B.\ Pushnitski and I.\ Sorrell for several useful discussions. 
This research was supported by EPSRC under grant EP/E037410/1.

\end{document}